\documentclass[a4paper,reqno]{amsart}

\usepackage{graphicx}
\usepackage{amsmath, amssymb}
\usepackage{graphicx}
\usepackage[colorlinks=true, allcolors=blue]{hyperref}

\usepackage{verbatim}

\usepackage{mathtools}

\usepackage{tikz}
\usepackage{tikz-cd}
\usepackage[all]{xy}

\usepackage{algorithm}
\usepackage{algpseudocode}

\makeatletter
\newcommand{\StatexIndent}[1][3]{%
  \setlength\@tempdima{\algorithmicindent}%
  \Statex\hskip\dimexpr#1\@tempdima\relax}
\makeatother

\tikzcdset{graphstyle/.append style={row sep=1.5em, column sep=1.5em, nodes={inner sep=0.5pt}, arrows={-stealth}}}

\usepackage{mathptmx}

\usepackage{blkarray,bigstrut}
\newcommand\topstrut[1][1.2ex]{\setlength\bigstrutjot{#1}{\bigstrut[t]}}
\newcommand\botstrut[1][0.9ex]{\setlength\bigstrutjot{#1}{\bigstrut[b]}}
\usepackage{multirow}
\usepackage{booktabs}

\usepackage{soul}

\usepackage{enumitem}
\setlist[enumerate]{labelindent = 1em,leftmargin=*,label=$(\arabic*)$}
\setlist[itemize]{labelindent = 1em,leftmargin=*,label=$\bullet$}
\numberwithin{equation}{section}

\usepackage{amsthm}
\newtheorem{theorem}{Theorem}[section]
\newtheorem{lemma}[theorem]{Lemma}
\newtheorem{corollary}[theorem]{Corollary}
\newtheorem{proposition}[theorem]{Proposition}

\theoremstyle{definition}
\newtheorem{definition}[theorem]{Definition}
\newtheorem{example}[theorem]{Example}

\newtheorem{remark}[theorem]{Remark}

\DeclareMathOperator{\rep}{rep}
\DeclareMathOperator{\mmod}{mod}
\DeclareMathOperator{\soc}{soc}
\DeclareMathOperator{\rad}{rad}

\DeclareMathOperator{\End}{End}

\DeclareMathOperator{\Hom}{Hom}

\DeclareMathOperator{\supp}{\mathrm{supp}}
\DeclareMathOperator{\Supp}{\mathrm{Supp}}
\DeclareMathOperator{\Tr}{Tr}
\DeclareMathOperator{\size}{Size}
\DeclareMathOperator{\Ker}{Ker}
\DeclareMathOperator{\Coker}{Coker}
\DeclareMathOperator{\ttop}{top}
\DeclareMathOperator{\rk}{rank}

\newcommand{\dimhom}[2]{\dim\Hom(#1,#2)}

\newcommand{\op}{\mathrm{op}}
\newcommand{\inv}{^{-1}}
\newcommand{\Sc}{\mathrm{Source}}

\newcommand{\full}{\operatorname{full}}
\newcommand{\Image}{\operatorname{Im}}

\newcommand\Ar[3]{\arrow[from={#1}, to={#2}, #3]}

\usepackage[e]{esvect}
\newcommand{\Af}[1]{{\vv*{A}{#1}}}
\newcommand{\Gf}[1]{{\vv*{G}{#1}}}

\newcommand{\Afn}{\Af{n}}

\newcommand{\II}[1]{\mathbb{I}_{#1}}

\newcommand{\intv}{\mathbb{I}}
\newcommand{\calS}{\mathcal{S}}

\newcommand{\card}{\#}

\newcommand{\calL}{\mathcal{L}}
\newcommand{\bbZ}{\mathbb{Z}}
\newcommand{\bbZnn}{\bbZ_{\ge 0}}
\newcommand{\bbN}{\mathbb{N}}

\DeclareMathOperator{\dimv}{\underline{dim}}

\usepackage{academicons}
\definecolor{orcidlogocol}{HTML}{A6CE39}

\newcommand{\suchthat}{\ifnum\currentgrouptype=16 \;\middle|\;\else\mid\fi}

\makeatletter
\let\save@mathaccent\mathaccent
\newcommand*\if@single[3]{%
  \setbox0\hbox{${\mathaccent"0362{#1}}^H$}%
  \setbox2\hbox{${\mathaccent"0362{\kern0pt#1}}^H$}%
  \ifdim\ht0=\ht2 #3\else #2\fi
  }
\newcommand*\rel@kern[1]{\kern#1\dimexpr\macc@kerna}
\newcommand*\widebar[1]{\@ifnextchar^{{\wide@bar{#1}{0}}}{\wide@bar{#1}{1}}}
\newcommand*\wide@bar[2]{\if@single{#1}{\wide@bar@{#1}{#2}{1}}{\wide@bar@{#1}{#2}{2}}}
\newcommand*\wide@bar@[3]{%
  \begingroup
  \def\mathaccent##1##2{%
    \let\mathaccent\save@mathaccent
    \if#32 \let\macc@nucleus\first@char \fi
    \setbox\z@\hbox{$\macc@style{\macc@nucleus}_{}$}%
    \setbox\tw@\hbox{$\macc@style{\macc@nucleus}{}_{}$}%
    \dimen@\wd\tw@
    \advance\dimen@-\wd\z@
    \divide\dimen@ 3
    \@tempdima\wd\tw@
    \advance\@tempdima-\scriptspace
    \divide\@tempdima 10
    \advance\dimen@-\@tempdima
    \ifdim\dimen@>\z@ \dimen@0pt\fi
    \rel@kern{0.6}\kern-\dimen@
    \if#31
      \overline{\rel@kern{-0.6}\kern\dimen@\macc@nucleus\rel@kern{0.4}\kern\dimen@}%
      \advance\dimen@0.4\dimexpr\macc@kerna
      \let\final@kern#2%
      \ifdim\dimen@<\z@ \let\final@kern1\fi
      \if\final@kern1 \kern-\dimen@\fi
    \else
      \overline{\rel@kern{-0.6}\kern\dimen@#1}%
    \fi
  }%
  \macc@depth\@ne
  \let\math@bgroup\@empty \let\math@egroup\macc@set@skewchar
  \mathsurround\z@ \frozen@everymath{\mathgroup\macc@group\relax}%
  \macc@set@skewchar\relax
  \let\mathaccentV\macc@nested@a
  \if#31
    \macc@nested@a\relax111{#1}%
  \else
    \def\gobble@till@marker##1\endmarker{}%
    \futurelet\first@char\gobble@till@marker#1\endmarker
    \ifcat\noexpand\first@char A\else
      \def\first@char{}%
    \fi
    \macc@nested@a\relax111{\first@char}%
  \fi
  \endgroup
}
\makeatother

\title{On Interval Decomposability of 2D Persistence Modules}

\author[H.~Asashiba]{Hideto Asashiba}
\author[M.~Buchet]{Micka\"{e}l Buchet}
\author[E.G.~Escolar]{Emerson G. Escolar}
\author[K.~Nakashima]{Ken Nakashima}
\author[M.~Yoshiwaki]{Michio Yoshiwaki}

\thanks{This work was partially supported by JST CREST Mathematics (15656429). M.Y. was partially supported by Osaka City University Advanced Mathematical Institute (MEXT Joint Usage/Research Center on Mathematics and Theoretical Physics JPMXP0619217849).}

\address[Hideto Asashiba]{Faculty of Science, Shizuoka University, and Center for Advanced Study, KUIAS, Kyoto University, Kyoto, Japan, and Osaka City University Advanced Mathematical Institute, Osaka, Japan}
\email{asashiba.hideto@shizuoka.ac.jp}

\address[Micka\"{e}l Buchet]{Institute of Geometry, TU Graz}
\email{buchet@tugraz.at}

\address[Emerson G. Escolar]{Graduate School of Human Development and Environment, Kobe University, Japan, and Center for Advanced Intelligence Project, RIKEN, Tokyo, Japan}
\email{e.g.escolar@people.kobe-u.ac.jp}

\address[Ken Nakashima]{Center for Advanced Intelligence Project, RIKEN, Tokyo, Japan}
\email{ken.nakashima@riken.jp}

\address[Michio Yoshiwaki]{Center for Advanced Intelligence Project, RIKEN, Tokyo, Japan, and Osaka City University Advanced Mathematical Institute, Osaka, Japan}
\email{michio.yoshiwaki@riken.jp}

\date{}

\begin{document}

\begin{abstract}
In the persistent homology of filtrations,
the indecomposable decompositions provide the persistence diagrams.
However, in almost all cases of multidimensional persistence,
the classification of all indecomposable modules is known to be a wild problem.
One direction is to consider the subclass of interval-decomposable persistence modules, which are direct sums of interval representations.
We introduce the definition of pre-interval representations, a more natural algebraic definition, and study the relationships between pre-interval, interval, and indecomposable thin representations. We show that over the ``equioriented'' commutative $2$D  grid, these concepts are equivalent.
Moreover, we provide a criterion for determining whether or not an $n$D persistence module is interval/pre-interval/thin-decomposable without having to explicitly compute decompositions. For $2$D persistence modules, we provide an algorithm together with a worst-case complexity analysis
that uses the total number of intervals in an equioriented commutative $2$D grid.
We also propose several heuristics to speed up the computation.

\end{abstract}

\maketitle

\section{Introduction}

In recent years, the use of topological data analysis to understand the shape of data has become popular, with persistent homology \cite{edelsbrunner2002topological} as one of its leading tools. Persistent homology is used to study the persistence -- the lifetime -- of topological features such as holes, voids, etc, in a filtration -- a one parameter increasing family of spaces. These features are summarized in a persistence diagram, a compact descriptor of 
the
birth and death parameter
values
of the topological features. This is enabled by the algebraic result of being able to decompose any $1$D persistence module into intervals \cite{carlsson2010zigzag,gabriel1972unzerlegbare}. The endpoints of these intervals are precisely the birth and death values of topological features.

The focus on one parameter families is a limitation of the current theory. While there is a need for practical tools  applying the ideas of persistence to multiparametric data, multidimensional persistence \cite{carlsson2009theory} is known to be difficult to apply practically and in full generality. More precisely, there does not exist a complete discrete invariant that captures all the indecomposable modules in this setting. This is unlike the $1$D case, where all indecomposables are guaranteed to be intervals 
and where the persistence diagram is a complete descriptor.
In terms of representation theory, this difficulty can be expressed by the fact that the commutative $n$D grid is of wild type (see \cite[Definition~6.4]{crawley1988tame})
for $n\geq 2$ (and grid large enough).

One way to avoid these difficulties is to consider persistence modules that decompose into indecomposables contained in a restricted set.
A promising candidate is the class of 
interval-decomposable persistence modules, which decompose into the so-called interval representations (see Definition~\ref{def:interval}).
For example, the paper \cite{dey2018computing} provides a polynomial-time algorithm for computing the bottleneck distance between two $2$D interval-decomposable persistence modules.
The paper \cite{bjerkevik2016stability} studies stability for certain subclasses of interval-decomposable modules.

In this work, we focus not only on interval representations, but also study some other related classes of indecomposable persistence modules. One reason is that the definition of interval representations used in the literature \cite{bjerkevik2016stability,botnan2016algebraic,dey2018computing} depends on a choice of bases and seems to be overly restrictive. For example, being an interval representation is not closed under isomorphisms. This is unsatisfying from an algebraic/category-theoretic point of view. We review the definition of thin representations, introduce the new notion of pre-interval representations, and study the relationship among thin, pre-interval, and interval representations.

As one contribution of this work, we answer the following question in Section~\ref{sec:thin_oracle}:
Given an $n$D persistence module, is there a way to determine, without explicitly computing its indecomposable decomposition, whether or not it is (pre)interval-decomposable or thin-decomposable? Given some set $\calS$ of indecomposable persistence modules, we provide in Theorem~\ref{thm:main_oracle} equivalent conditions for determining $\calS$-decomposability. 
In the case that $\calS$ is a finite set, this translates into an implementable criterion.

In Section~\ref{sec:thins_and_intervals}, we focus on the equioriented commutative $2$D grid. It is clear that over a $1$D grid (i.e.~the quiver $\Af{n}$, see Section~\ref{sec:background}), being a thin indecomposable is equivalent to being isomorphic to an interval representation, since each indecomposable is isomorphic to an interval \cite{gabriel1972unzerlegbare}, and conversely, interval representations are automatically thin and indecomposable in general.
In subsection \ref{subsec:results_2dthins}, we show that this relationship holds also in the equioriented commutative $2$D grid: any thin indecomposable is isomorphic to an interval representation. In subsection~\ref{subsec:examples}, we give examples for a non-equioriented commutative $2$D grid and for an equioriented commutative $3$D grid showing that this relationship does not hold in general. Finally, we provide a count of the total number of intervals in an equioriented commutative $2$D grid in Theorem~\ref{thm:num-Imn} by relating intervals in this setting to the so-called parallelogram polyominos.

In Section~\ref{sec:cc}, we provide a detailed algorithm (Algorithm~\ref{alg:overall})
for determining interval-decomposability, based on Theorem~\ref{thm:main_oracle}, and give its computational complexity. 
In particular, we give detailed descriptions of 
the computation of almost split sequences ending at interval representations and of dimensions of homomorphism spaces, which are used to compute multiplicities of interval summands.
Furthermore, we propose several heuristics to reduce the number of interval representations to be checked.

Related to the question of determining interval-decomposability, we note the following  results. Previous works~\cite{cochoy2020decomposition,botnan2020rectangle} show that a pointwise finite-dimensional persistence module satisfies a certain local property called exactness if and only if it is rectangle-decomposable. Rectangles are intervals, and thus this result gives a criterion for a restricted class of interval-decomposables. 
Unfortunately, no such local criterion exists for interval-decomposability~\cite{botnan2020local}.
We note that our criterion in Theorem~\ref{thm:main_oracle} is not local, as it relies on 
the
computation of dimensions of certain homormophism spaces.

In their paper \cite{dey2019generalized}, Dey and Xin gave an algorithm to decompose a restricted class of $n$D persistence modules $M$. Their algorithm proceeds on the assumption that the module is ``distinctly graded''. One formulation of this condition is that there exists a projective presentation
$P_1 \to P_0 \to M \to 0$ 
of $M$ with both $P_0$ and $P_1$ square-free\footnote{A square-free module is a direct sum of non-isomorphic indecomposables.} modules.
Furthermore, in version 5 of their arXiv preprint, they claim that their ``algorithm can be applied to determine whether a persistence module is interval decomposable'' \cite{dey2019generalized}.
When the module is not distinctly graded, one can arbitrarily fix an order on the grades. The number of possible orders is finite, and they claim that at least one of those orders provide a full decomposition of the module, and therefore it is enough to test all possible orders.
However we argue that this claim is erroneous.
We provide a counter-example by giving an interval-decomposable module $M$ that is not distinctly graded and such that there exists no order on the grades that leads to a full decomposition by applying their algorithm.


\section{Background}
\label{sec:background}

\subsection{Quivers and their representations}
We use the language of the representation theory of bound quivers.
For more details, we refer the reader to the book \cite{assem2006elements}, for example.
Let us recall some basic definitions.

A \emph{quiver} is a quadruple $Q=(Q_0, Q_1, s, t)$ of sets $Q_0, Q_1$ and maps $s, t \colon Q_1 \to Q_0$.
If we draw each $a \in Q_1$ as an arrow $a\colon s(a) \to t(a)$, then $Q$ can be presented as a directed graph.
Then we call elements of $Q_0$ (resp.\ elements of $Q_1$, $s(a)$ and $t(a)$) vertices of $Q$ (resp.\ arrows of $Q$, the source of $a$ and the target of $a$ for each $a \in Q_1$). Let $n$ be a positive integer. We denote by $\Af{n}$ the quiver presented as the directed graph
\[
  \begin{tikzcd}
    1 \rar & 2 \rar & \cdots \rar & n
  \end{tikzcd}.
\]
The quiver $\Af{n}$ plays a central role in persistence theory.

A \emph{subquiver} $Q'$ of a quiver $Q$ is a quiver $Q' = (Q'_0, Q'_1,s',t')$ such that $Q'_0 \subseteq Q_0$, $Q'_1 \subseteq Q_1$, and $s'(a) = s(a)$, $t'(a)=t(a)$ for all $a\in Q'_1$.
A subquiver $Q'$ is said to be \emph{full} if it contains all arrows of $Q$ between all pairs of vertices in $Q'$.

A {\em quiver morphism} from a quiver $Q$ to a quiver $Q'$ is a pair $(f_0, f_1)$ of maps $f_0\colon Q_0 \to Q'_0$ and
$f_1\colon Q_1 \to Q'_1$ such that $f_0 s = s'f_1, f_0 t = t' f_1$.
A {\em path} from a vertex $x$ to a vertex $y$ of length $n\ (\ge 1)$ in $Q$ is a sequence $\alpha_n \cdots \alpha_2 \alpha_1$ of arrows $\alpha_1, \alpha_2 \dots, \alpha_n$ of $Q$ such that $s(\alpha_1) = x$, $t(\alpha_n) = y$, and
$s(\alpha_{i+1}) = t(\alpha_i)$ for all $1 \le i \le n-1$.
Here we call $x$ and $y$ the {\em source} and the {\em target} of this path, respectively.
Note that this can be viewed as a quiver morphism  $f:\Af{n+1} \rightarrow Q$, with $f(1) = x$ and $f(n+1)=y$.

Next, we give some definitions concerning convexity and connectedness in quivers.
\begin{definition}[{\cite[p.~303]{assem2006elements}}, Convex subquiver]
  Let $Q$ be a quiver. A full subquiver $Q'$ of $Q$ is said to be \emph{convex} in $Q$ if and only if for all vertices $x$, $y$ in $Q'_0$, and for all paths $p$  from $x$ to $y$ in $Q$, all vertices of $p$ are in $Q'_0$ (and thus $p$ is a path in $Q'$).
\end{definition}

\begin{definition}[Connected]
A quiver $Q$ is said to be \emph{connected} if
it is connected as an undirected graph,
namely, if for each pair $x, y$ of vertices of $Q$ there exists a quiver $W$ with underlying graph of the form
$\xymatrix@1@C=10pt{1 \ar@{-}[r] &2 \ar@{-}[r] &\cdots \ar@{-}[r]&n}$ for some $n\ (\ge 1)$ and a quiver morphism $f\colon W \to Q$ such that $f(1) = x, f(n) = y$.
\end{definition}

We give the following definition of intervals in a quiver.

\begin{definition}[Interval subquiver]
\label{def:interval}
  Let $Q$ be a quiver. An {\em interval} of $Q$ is a convex and connected subquiver 
of $Q$.
\end{definition}

This definition is a generalization of the one \cite{botnan2016algebraic,dey2018computing} for commutative grids used in persistence theory. This in turn generalizes intervals of $\Af{n}$ in the usual sense: It is clear that an interval subquiver of $\Af{n}$ is a full subquiver containing all vertices $i$ for $b\leq i \leq d$, for some $b\leq d \in \mathbb{N}$. The interest in intervals comes mainly from the intuition about $\Af{n}$ in persistence theory: they form the building blocks of representations of $\Af{n}$, are simple to describe (parameters $b$ and $d$ only), and have a useful interpretation as the births and deaths of topological features.

Throughout this work, we let $K$ be a field, and $Q$ a quiver.
Paths in $Q$ are
said to be {\em parallel} if 
they have the same source and the same target.
A {\em relation} is a $K$-linear combination of parallel paths of length at least 2. In what follows, we need the concept of \emph{bound quivers}, which we denote by $(Q,R)$ for a quiver $Q$ with a set of relations $R$. First, we define the following special set of relations.

\begin{definition}[Full commutativity relations]
  Let $Q$ be a quiver. The set of \emph{full commutativity relations} of $Q$ is
  \[
    R = \left\{p_1 - p_2 \suchthat p_1,p_2 \text{ are parallel paths of length $\ge 2$ in $Q$
    }
\right\}.
  \]
\end{definition}

 Recall that a \emph{$K$-representation} of $Q$ is a family $V = (V(x), V({\alpha}))_{x\in Q_0, \alpha \in Q_1}$,
where $V(x)$ is a $K$-vector space for each vertex $x$, and $V(\alpha):V(x)\rightarrow V(y)$ is a $K$-linear map for each arrow $\alpha:x\rightarrow y$. For example, the zero representation is $0=(0,0)_{x\in Q_0, \alpha \in Q_1}$.

A morphism $f:V\to W$ is a family $(f_x)_{x\in Q_0}$ of $K$-linear maps $f_x:V(x)\to W(x)$ such that $W(\alpha)f_x=f_y V(\alpha)$ for each arrow $\alpha: x \to y$ in $Q_1$. 
A \emph{subrepresentation} $W$ of $V$ is a representation $W=(W(x), W({\alpha}))_{x\in Q_0, \alpha \in Q_1}$ such that $W(x)$ is a vector subspace of $V(x)$ for each $x\in Q_0$ and the collection of inclusions $W(x) \to V(x)$ forms a morphism $W \to V$.
The direct sum $V \oplus W$ of representations $V = (V(x), V({\alpha}))_{x\in Q_0, \alpha \in Q_1}$ and $W = (W(x), W({\alpha}))_{x\in Q_0, \alpha \in Q_1}$ is the representation $(V(x)\oplus W(x), V({\alpha})\oplus W ({\alpha}))_{x\in Q_0, \alpha \in Q_1}$.
A representation $V$ is \emph{indecomposable} if $V\cong V_1 \oplus V_2$ implies $V_1=0$ or $V_2=0$.  
The \emph{dimension} of a representation $V$ is defined to be $\dim V = \sum_{x\in Q_0}\dim V(x)$.
We call a representation $V$ \emph{finite-dimensional} if $\dim V < \infty$.

A representation $V$ of $Q$ is said to be a \emph{representation of the bound quiver} $(Q,R)$
if $V$ satisfies the relations given by a set of relations $R$ (i.e.,~if $\sum_{i=1}^n t_iV(\mu_i) = 0$ for all $\sum_{i=1}^n t_i \mu_i \in R$ with $\mu_i$ paths and $t_i \in K$, where $V(\mu)=V(\alpha_m)\cdots V(\alpha_1)$ for a path $\mu=\alpha_m\cdots \alpha_1$).
The category of finite-dimensional $K$-representations of $(Q,R)$ will be denoted by $\rep_K(Q,R)$. In this work, we consider only finite-dimensional representations. 

As an example, given $1\leq b\leq d\leq n$, the interval representation $\intv[b,d]$ of $\Af{n}$ is the representation
\[
  \intv[b, d] \colon 0 \longrightarrow \cdots
  \longrightarrow 0
  \longrightarrow \overset{b\text{-th}}{K} \longrightarrow K
  \longrightarrow \cdots
  \longrightarrow \overset{d\text{-th}}{K}
  \longrightarrow 0 \longrightarrow \cdots
  \longrightarrow 0,
\]
which has the vector space $\intv[b,d](i) = K$ at the vertices $i$ with $b\leq i\leq d$, and $0$ elsewhere, and where the maps between the neighboring vector spaces $K$ are identity maps and zero elsewhere. It is known that $\left\{\intv[b,d]\right\}_{1\leq b \leq d \leq n}$
gives a complete list of indecomposable representations of $\Afn$, up to isomorphisms (see \cite{gabriel1972unzerlegbare}).

\begin{definition}[{\cite[p.~93]{assem2006elements}}, Support]
  Let $V \in \rep_K(Q,R)$. The \emph{support} of $V$, denoted by $\supp(V)$, is the full subquiver of $Q$ consisting of the vertices $x$ with $V(x) \neq 0$.

\end{definition}
Continuing our example, the support $\supp\intv[b,d]$ of the interval representation $\intv[b,d]$ is clearly the full subquiver with vertices $i$ for $b\leq i\leq d$, which is an interval subquiver of $\Af{n}$ in the sense of Definition~\ref{def:interval}.

Below, we define the equioriented grid by taking a product of $\Afn$. First, we give the general definition of products of quivers.
\begin{definition}[Products of quivers]
  Let $Q = ( Q_0, Q_1, s,t) $ and $Q' = (Q'_0, Q'_1,s',t')$ be quivers.
  \begin{itemize}
    \item The Cartesian product $Q \times Q'$ is the quiver with the set of vertices $Q_0 \times Q'_0$ and the set of arrows
$\{(x, a'), (a, x')\mid x \in Q_0, x' \in Q'_0, a \in Q_1, a' \in Q'_1\}$,
where the sources and targets are determined by
\[
\begin{aligned}
(a,x')&\colon (x,x') \to (y,x')
\text{ if } a \colon x \to y, \\
(x, a')&\colon (x, x') \to (x, y')
\text{ if } a' \colon x' \to y'.
\end{aligned}
\]
    \item The tensor product $Q \otimes Q'$ is the bound quiver $Q \times Q'$ with the commutativity relations
$(a,y')(x,a') - (y,a')(a,x')$ for all arrows $a \colon x \to y$ in $Q$ and $a'\colon x' \to y'$ in $Q'$.
   \end{itemize}
 \end{definition}

\begin{definition}[Equioriented commutative grid]
  Let $m, n$ be positive integers. The bound quiver $\Gf{m,n} = \Af{m}\otimes\Af{n}$, which is the $2$D grid of size $m\times n$ with all arrows in the same direction and with full commutativity relations, is called the \emph{equioriented commutative grid} of size $m\times n$.
\end{definition}

  In this work, we use the convention of displaying $\Gf{m,n}$ as a $2$D grid with $m$ columns and $n$ rows, with arrows pointing right or up.
For example, $\Gf{4,3} =\Af{4}\otimes \Af{3}$ is the quiver
\[
  \begin{tikzcd}[row sep=1em, column sep=1em]
    \bullet \rar & \bullet \rar & \bullet \rar & \bullet \\
    \bullet \rar\uar & \bullet \rar\uar & \bullet\uar\rar & \bullet\uar \\
    \bullet \rar\uar & \bullet \rar\uar & \bullet\uar\rar & \bullet\uar
  \end{tikzcd}
\]
with full commutativity relations.

Alternatively, $\Gf{m,n} = \Af{m} \otimes \Af{n} = (\Af{m}\times \Af{n}, R)$, where $R$ is the 
full commutativity relations, can be understood using the tensor product of path algebras. That is, we have $K(\Af{m}\otimes\Af{n}) \cong K\Af{m}\otimes_K K\Af{n}$ as algebras, where the notation $K(\Af{m}\otimes\Af{n})$ is the quotient $K(\Af{m}\times \Af{n})/\langle R\rangle$ of the path algebra by the two-sided ideal generated by $R$.

While not a focus of this paper, we define the equioriented commutative $n$D grid of size $m_1\times m_2\times \hdots \times m_n$ as $\Gf{m_1,\hdots,m_n} \doteq \Af{m_1}\otimes\hdots \otimes \Af{m_n}$. Similarly, non-equioriented versions of the commutative $n$D grid can be defined by taking the tensor product of $A_{m_i}$-type quivers, where for at least one $i$, the arrows in $i$th factor are not pointing in the same direction.

\subsection{Representations of interest}
Throughout this section, we let $(Q,R)$ be a bound quiver. We first start with the following straightforward definition.
\begin{definition}[Thin representations]
  A representation $V \in \rep_K(Q,R)$ is \emph{thin} if $\dim_K V(x) \leq 1$ for each vertex $x$ of $Q$.
\end{definition}
Note that we do not require indecomposability for thin representations. If $V$ is thin and indecomposable, we say that $V$ is a thin indecomposable. Next, we provide our definition of (pre-)interval representations of a general (bound) quiver.

\begin{definition}[Interval and pre-interval representations]
  \label{def:interval_rep}
  \leavevmode
  \begin{enumerate}
  \item A representation $V \in \rep_K(Q,R)$ is an \emph{interval representation} if and only if
    \begin{itemize}
    \item (Thinness) it is thin, and
    \item (Interval support) its support $\supp(V)$ is an interval of $Q$, and
    \item (Identity over support) for all arrows $\alpha \in \supp(V)$, $V(\alpha)$ is an identity map.
    \end{itemize}
    Note that this definition is not stable under isomorphism (see Remark~\ref{rem:interval}). Thus, in this work, by interval representation we also mean ``isomorphic to an interval representation'' if there is no risk of confusion.
    
  \item If, instead of the third condition (Identity), $V$ satisfies the condition
    \begin{itemize}
    \item (Nonzero over support) for all arrows $\alpha \in \supp(V)$, $V(\alpha)$ is nonzero,
    \end{itemize}
    then $V$ is said to be a \emph{pre-interval representation}.
  \end{enumerate}
\end{definition}

Recall that the support of a representation $V$ is the full subquiver of vertices $x$ with $V(x) \neq 0$. Thus, the ``identity/nonzero over support'' conditions means that if $V(x)$ and $V(y)$ are nonzero, then \emph{all} arrows $\alpha:x\rightarrow y$ have 
$V(\alpha)$
identity or nonzero, respectively.

\begin{remark}\leavevmode
  \label{rem:interval}
  \begin{enumerate}
  \item The condition ``identity over support'' implies that $V(x)$ and $V(y)$ are equal as (one-dimensional) vector spaces. This condition is not stable under isomorphisms.
    For example, consider $\Af{2}$ and its $\mathbb{R}$-representations
    \[
      V :
      \begin{tikzcd}
        \mathbb{R} \rar{1} & \mathbb{R}
      \end{tikzcd}
      \text{ and }
      V' :
      \begin{tikzcd}
        \mathbb{R} \rar{f} & \mathbb{R}a
      \end{tikzcd}
    \]
    where $f$ is the linear map determined by taking $f(1) = a$.
    Then, $V \cong V'$ and are both pre-interval. Clearly, $V$ is interval, but $V'$ is not. 
    The one-dimensional vector spaces $\mathbb{R}$ and $\mathbb{R}a$ are not equal, only isomorphic. 
    
\item In Section~\ref{subsec:examples} we give examples where the three classes (thin indecomposable, pre-interval, isomorphic to an interval) are not equal.

  \item Under thinness and interval support conditions, a representation $V$ is isomorphic to an interval representation if and only if there exist bases $v_x \in V(x)$ for all $x \in \supp(V)_0$ such that the following holds:
    \begin{equation}\label{eq:iso-interval}
      V(\alpha)(v_x) = v_y
      \text{
        for each arrow $\alpha\colon x \to y$ in $\supp(V)$.}
    \end{equation}
    This condition will be used to show that a representation is isomorphic to an interval representation.

  \item
    If the coefficient field is $K = \mathbb{F}_2$, then every pre-interval representation is isomorphic to an interval representation.
    We note that in topological data analysis it is indeed common to choose the base field $\mathbb{F}_2$. Thus, it may seem that there is no need to consider the pre-interval representations. 
    However, we note the following two reasons for considering fields other than $\mathbb{F}_2$.
    
    First, homology over $\mathbb{F}_2$ does not capture topological torsion.
    Therefore, working with other fields provides more information.
    Second, decomposition of representations over $\mathbb{F}_2$ presents some deep algebraic complications, in the representation-infinite setting. An intuition into these complications can be obtained by contrasting the following two canonical forms arising in matrix decompositions. The Jordan canonical form, available over algebraically closed fields, is relatively simple compared to the rational canonical form, which involves irreducible polynomials in general. 
    In this setting, decomposition over an infinite field (the algebraic closure) involves simpler summands.
  \end{enumerate}
\end{remark}

We note that our definition generalizes the usual definition of intervals and interval representations in the literature. For example, \cite{bjerkevik2016stability} and \cite{botnan2016algebraic} defines intervals and interval representations over posets in general, and \cite{dey2018computing} over the poset $\bar{\mathbb{R}}^n := (\mathbb{R}\cup\{\infty\})^n$. It is clear that, given a poset $P$, we can construct an acyclic quiver with full commutativity relations $(Q,R)$, and vice-versa, such that
$\rep_K(Q,R)$
is equivalent to the category of  pointwise finite-dimensional $K$-linear representations of $P$.

Then, it can be checked that an interval $\emptyset \neq I \subset P$ in the sense of \cite{bjerkevik2016stability,botnan2016algebraic} corresponds to a nonempty interval $I$ in the sense of our Definition~\ref{def:interval}. In this setting, convexity corresponds to the condition that $a,c\in I$ and $a\leq b\leq c$ implies $b\in I$. On the other hand, connectedness
corresponds to the condition that for any $a,c\in I$, there is a sequence $a=x_0,x_1,\hdots,x_\ell=c$ in $I$ with $x_i$ and $x_{i+1}$ comparable for all $0\leq i \leq \ell-1$. Similarly, given an interval $J$, the interval module $I^J$ as defined in Definition~2.1 of \cite{botnan2016algebraic} is precisely an interval representation in the sense of our Definition~\ref{def:interval_rep} with support $J$.

\begin{lemma}[See also {\cite[Prop.~2.2]{botnan2016algebraic}}]
  Let $V$ be a nonzero representation of $(Q,R)$. If $V$ is an interval or pre-interval representation, then $V$ is indecomposable.
\end{lemma}
\begin{proof}
  The proof is similar to the proof of Prop.~2.2 in {\cite{botnan2016algebraic}}. If $V$ is an interval or pre-interval representation, then it is also thin, so  without loss of generality, we assume that the vector spaces of $V$ are $K$ or $0$. Then, endomorphisms of $V$ act at each vertex by multiplication by some scalar. By commutativity requirements on endomorphisms together with the ``nonzero over support'' condition, each pair of these scalars over vertices in the same connected component are equal. Thus, by connectedness, $\End(V) \cong K$, and hence $V$ is indecomposable.
\end{proof}

In general, we have the following hierarchy of these classes of indecomposable representations:
\begin{equation}
  \label{eq:hierarchy}
  \{V \suchthat V \cong \text{ an interval}\} \subset \{V \suchthat V \text{ pre-interval}\} \subset \{V \suchthat V \text{ thin indecomposable}\}.
\end{equation}
Later, we shall show that for the equioriented commutative 2D grid, these three collections are equal. We shall also provide examples of where the inclusions are strict in the general case.

Finally, we provide the following definitions concerning these special classes of indecomposables.
\begin{definition}
  Let $(Q,R)$ be a bound quiver.
  \begin{enumerate}
  \item A representation $V \in \rep_K(Q,R)$ is said to be \emph{interval-decomposable} (resp.\ \emph{pre-interval-decomposable}, \emph{thin-decomposable}) if and only if each direct summand in some indecomposable decomposition of $V$ is an interval representation (resp.\ pre-interval representation, thin representation).
  \item The bound quiver $(Q,R)$ itself is said to be \emph{interval-finite} (resp.\ \emph{pre-interval-finite}, \emph{thin-finite}) if and only if the number of isomorphism classes of its interval representations (resp. pre-interval representations, thin indecomposables) is finite.
  \end{enumerate}
\end{definition}

In the rest of this work, we consider only bound quivers $(Q,R)$ such that $KQ/\langle R \rangle$ is a finite-dimensional $K$-algebra. This holds, for example, if $\langle R \rangle$ is an admissible ideal, or if $Q$ is a finite acyclic quiver.
With this assumption, we can use the Auslander-Reiten theory needed for the next section. Furthermore, we fix a complete set $\calL$ of representatives of isomorphism classes of (finite-dimensional) indecomposable representations of $(Q, R)$,
which we identify with the set of vertices of the Auslander-Reiten quiver of $(Q, R)$. For more details on the Auslander-Reiten theory, we refer the reader to the books~\cite{assem2006elements,auslander1997representation}.

\subsection{Decomposition theory} \label{subsection:dec}

We consider only bound quivers $(Q,R)$ such that $KQ/\langle R \rangle$ is a finite-dimensional $K$-algebra. First recall the Krull-Schmidt Theorem, which can be stated as follows.

\begin{theorem}[Krull-Schmidt]
For each representation $M$ of $(Q, R)$ there exists a unique function $d_M \colon \calL \to \bbZnn$ such that
$M \cong \bigoplus_{L \in \calL} L^{d_M(L)}$.
Therefore for each pair $M, N$ of representations of
$(Q, R)$ we have $M \cong N$ if and only if
$d_M = d_N$.
\end{theorem}

In this subsection, let us review decomposition theory \cite{Asashiba2017,dowbor2007multiplicity} which gives an algorithm to compute the multiplicity $d_M(L)$ for all $L \in \calL$ by using Auslander-Reiten theory.
For the details of Auslander-Reiten theory, we refer the reader to \cite[Chapter IV]{assem2006elements} or \cite[Chapter V]{auslander1997representation}.

Here, we briefly provide the definitions required for Theorem~\ref{thm:PD-ass}
and its dual.

For a representation $M$ recall that the sum of all simple submodules of $M$ is called the \emph{socle} of $M$,
denoted by $\soc M$, and that the intersection of the kernels
of all homomorphisms from $M$ to simple modules is called
the \emph{radical} of $M$, denoted by $\rad M$.
We set $\ttop M:= M/\rad M$, and call it the \emph{top} of $M$.
Note that $\ttop P$ of an indecomposable projective representation
$P$ and $\soc I$ of an indecomposable injective representation $I$ are simple.

\begin{definition}
Let $f \colon X \to Y$ be a morphism of representations of $(Q, R)$.
\begin{enumerate}
\item
$f$ is said to be {\em left minimal} (resp.\ {\em right minimal})
if for any morphism $h\in \End (Y)$ (resp.\ $h \in \End (X)$)
$hf = f$ (resp.\ $fh = f$) implies that $h$ is an automorphism.
\item
A non-section (resp.\ non-retraction) $f$ is said to be {\em left almost split}
(resp.\ {\em right almost split})
if for every non-section $u \colon X \to M$
(resp.\ non-retraction $u \colon M \to Y$)
there is a morphism $v \colon Y \to M$ such that $vf = u$
(resp.\ a morphism $v \colon M \to X$ such that $fv = u$).
\item
$f$ is called a {\em source map} (resp. {\em sink map})
if $f$ is both left minimal and left almost split
(resp.\ right minimal and right almost split).
\end{enumerate}
\end{definition}

A short exact sequence $0 \to X \xrightarrow{f} Y \xrightarrow{g} Z \to 0$
is an \emph{almost split sequence} if $f$ is a source map and $g$ is a sink map. 
Source maps and sink maps from every indecomposable representation is given
as follows, which is a fundamental theorem in Auslander-Reiten theory:

\begin{theorem}
Let $L$ be an indecomposable representation of $(Q, R)$,
$f \colon L \to U$ a source map and $g \colon V \to L$ a sink map.
Then
\begin{enumerate}
\item
If $L$ is injective, then $f$ is given by the composite of the canonical epimorphism
$L \to L/\soc L$ followed by an isomorphism $L/\soc L \to U$.
In particular, $U \cong L/\soc L =: {}_LE$.
\item
If $L$ is not injective, then $f$ is given by the composite of $\alpha$ followed by
an isomorphism ${}_LE \to U$, where
\[
0 \to L \xrightarrow{\alpha} {}_LE \xrightarrow{\beta} \tau\inv L \to 0
\]
is an almost split sequence.
In particular, $U \cong {}_LE$.
\item
If $L$ is projective, then $g$ is given by the composite of an isomorphism
$V \to \rad L$ followed by the inclusion map $\rad L \to L$.
In particular, $V \cong \rad L =: E_L$.
\item
If $N$ is not projective, then $g$ is given by the composite of an isomorphism
$V \to E_L$ followed by $\beta$, where
\[
0 \to \tau L \xrightarrow{\alpha} E_L \xrightarrow{\beta} L \to 0
\]
is an almost split sequence.
In particular, $V \cong E_L$.
\end{enumerate}
Here,
$\tau:= D \Tr, \tau\inv:= \Tr D$, where $\Tr$
denotes
the transpose (see \cite[Chapter IV.2]{assem2006elements}).
\end{theorem}

For each indecomposable representation $L$ of $(Q, R)$
we can decompose ${}_LE$ and $E_L$ as
\[
{}_LE = \bigoplus_{X \in J_L} X^{(a_L(X))},
\quad
E_L = \bigoplus_{X \in K_L} X^{(b_L(X))}
\]
for a unique subset $J_L$ (resp.\ $K_L$) of $\calL$ and a unique function $a_L \colon J_L \to \bbZ_{>0}$ 
(resp.\ $b_L \colon K_L \to \bbZ_{>0}$).
Recall
that $L$ is injective if and only if $\tau\inv L = 0$
(dual version of \cite[Prop.\ IV.1.10(b)]{auslander1997representation}),
and that $L$ is projective if and only if $\tau L = 0$.


\begin{theorem}[{\cite[Thm.~3]{Asashiba2017}}, {\cite[Cor.~2.3]{dowbor2007multiplicity}}]\label{thm:PD-ass}
Let $M \in \rep_K(Q,R)$ and $L \in \calL$. Then
$d_M(L)$ is computed by the following four formulae:
\begin{align}
d_M(L) &= \dimhom{L}{M} - \dimhom{{}_LE}{M} + \dimhom{\tau\inv L}{M},\\
d_M(L) &= \dimhom{L}{M} - \sum\limits_{X\in J_L}a_L(X)\dimhom{X}{M} + \dimhom{\tau\inv L}{M},\label{eq:d-M-L-starting}\\
d_M(L) &= \dimhom{M}{\tau L} - \dimhom{M}{E_L} + \dimhom{M}{L},
\label{eq:dual-a}\\
d_M(L) &= \dimhom{M}{\tau L} - \sum\limits_{X\in K_L}b_L(X)\dimhom{M}{X} + \dimhom{M}{L}.
\label{eq:dual-b}
\end{align}
\end{theorem}

For an indecomposable representation $X$ of $(Q, R)$ the function
$$
s_X:= \dim \Hom(X, \text{-})\colon \rep(Q, R) \to \bbZnn,\quad
M \mapsto \dim \Hom(X, M)
$$
is called the {\em starting function} from $X$. Dually,
$$
t_X:= \dim \Hom(\text{-}, X)\colon \rep(Q, R) \to \bbZnn,\quad
M \mapsto \dim \Hom(M, X)
$$
is called the {\em stopping function} to $X$.
Using these, the formulae~\eqref{eq:d-M-L-starting} and~\eqref{eq:dual-b} have the following forms:
\begin{equation}\label{eq:PD-starting}
	d_M(L) = s_{L}(M) - \sum\limits_{X\in J_L}a_L(X)s_{X}(M) + s_{\tau\inv L}(M).
\end{equation}
\begin{equation}\label{eq:PD-stopping}
	d_M(L) = t_{\tau L}(M) - \sum\limits_{X\in K_L}b_L(X)t_{X}(M) + t_{L}(M).
\end{equation}
Note that the value of $s_X(M)$ (or $t_X(M)$) can be computed as the rank of some matrix defined by $M$ for each $X$ (see \cite{Asashiba2017} for details).

For completeness we added the dual versions \eqref{eq:dual-a} and \eqref{eq:dual-b},
which were not presented in \cite[Thm.~3]{Asashiba2017}.
Later we will use formula \eqref{eq:dual-a} 
to examine the computational complexity of our algorithm for determining interval-decomposability.


\section{Determining $\calS$-decomposability}
\label{sec:thin_oracle}

To state our theorem, we 
first
generalize the idea of interval-decomposability and thin-decomposability in the following way. Let $\calS$ be a subset of 
the chosen complete set $\calL$ of representatives of isomorphism classes of indecomposable representations.
Then, $M \in \rep(Q,R)$ is said to be \emph{$\calS$-decomposable} if and only if $M \cong \bigoplus_{L \in \calS}L^{d_M(L)}$.
In this section, we use the decomposition theory to determine whether or not a given persistence module is $\calS$-decomposable, provided $\calS$ is finite.

\begin{theorem}
  \label{thm:main_oracle}
  Let $\calS$ be a subset of $\calL$, and $M \in \rep (Q,R)$.
Then the following are equivalent.
\begin{enumerate}[]
\item
$M$ is $\calS$-decomposable;
\item
$\dim M = \sum\limits_{L\in \calS} d_M(L) \dim L$; 
\item
$\dim M + \sum\limits_{L\in \calS} \sum\limits_{X\in J_L}a_L(X)s_{X}(M) \dim L= \sum\limits_{L\in \calS}(s_{L}(M) + s_{\tau\inv L}(M)) \dim L$;
\item 
$\dim M + \sum\limits_{L\in \calS} \sum\limits_{X\in K_L}b_L(X)t_{X}(M) \dim L= \sum\limits_{L\in \calS}(t_{\tau L}(M) + t_{L}(M)) \dim L$.
\end{enumerate}
\end{theorem}
\begin{proof}
The isomorphism 
$M \cong 
\left(\bigoplus\limits_{L \in \calS}L^{d_M(L)}\right) 
\oplus
\left(\bigoplus\limits_{L \in \calL \setminus\calS}L^{d_M(L)}\right)$
shows that
\[
\dim M = \sum_{L \in \calS}d_M(L) \dim L + \sum_{L \in \calL\setminus\calS}d_M(L) \dim L
\]
Then we have equivalences (1) $\Leftrightarrow$ $\bigoplus_{L \in \calL \setminus\calS}L^{d_M(L)} =0$
$\Leftrightarrow$ $\sum_{L \in \calL\setminus\calS}d_M(L) \dim L = 0$
$\Leftrightarrow$ (2).
The equivalences (2) $\Leftrightarrow$ (3) $\Leftrightarrow$  (4) follow from Equations~\eqref{eq:PD-starting}~and~\eqref{eq:PD-stopping}.
\end{proof}

%
%
%
%
In the case that $\calS$ is finite, Theorem~\ref{thm:main_oracle} gives us a criterion to determine the $\calS$-decomposability of a given $M \in \rep (Q,R)$. In particular, we only need to consider a finite number of values $d_M(X)$ for $X\in \calS$ and then compare $\dim M$ with $\sum_{X\in \calS} d_M (X) \dim X$. If these values are equal, then the given $M\in \rep (Q,R)$ is $\calS$-decomposable by the implication (2) $\Rightarrow$ (1).
The formula (3) gives a criterion for $M$ to be $\calS$-decomposable
by using the function $\dim$ and the values $s_X(M)$ of starting functions from indecomposable representations
$X \in \calS \cup (\bigcup_{L \in \calS}J_L)$, on which the computation of
$d_M(L)$ depends.

Thus, it is important to determine whether or not a particular bound quiver is thin-finite or (pre-)interval-finite. In Section~\ref{sec:thins_and_intervals}, we study the equioriented  commutative 2D grid. Here, we give the following trivial observations of some settings where the criterion given by Theorem~\ref{thm:main_oracle} can be immediately applied.
\begin{lemma}
\label{lem:finfield}
Let $Q$ be a finite $($bound\/$)$ quiver, and $K$ a finite field. Then $Q$ is thin-finite.
\end{lemma}
\begin{proof}
Consider the number of possible thin representations of $Q$. Since $Q$ has a finite number of arrows, and over each arrow, a thin representation $V$ can only have (up to isomorphism) $f:K\rightarrow K$ or $K\rightarrow 0$ or $0\rightarrow K$, where there are only a finite number of possibilities for $f \in \Hom(K,K)\cong K^\op$. Thus the number of possible thin representations (up to isomorphism) of $Q$ is finite.
\end{proof}
Note that because of the hierarchy in Ineq.~\eqref{eq:hierarchy}, thin-finiteness implies pre-interval-finiteness. For finite quivers, interval-finiteness is automatic, as the next lemma shows.
\begin{lemma}
Let $Q$ be a finite $($bound\/$)$ quiver. Then $Q$ is interval-finite.
\end{lemma}
\begin{proof}
This follows by a similar counting argument for interval representations $V$ as in the previous lemma, but this time the only possibility for $f:K\rightarrow K$ is the identity since $V$ is an interval representation. Note that $K$ being a finite field is not required.
\end{proof}


\section{Equioriented commutative $2$D grid}
In this section, we focus our attention on the equioriented commutative $2$D grid $\Gf{m,n}$.
We show that each thin indecomposable of $\Gf{m,n}$ is isomorphic to an interval representation and enumerate all interval representations of $\Gf{m,n}$.
\label{sec:thins_and_intervals}
\subsection{2D 	thin indecomposables are interval representations }
\label{subsec:results_2dthins}

First, let us show that interval subquivers of $\Gf{m,n}$ can only have a ``staircase'' shape. To make this more precise, we define the following.

Let $m$ and $n$ be fixed positive integers, and let $\II{m,n}$ be the set of all nonempty interval subquivers of  $\Gf{m,n}$. For $1\leq j \leq n$, a \emph{slice} at row $j$ is a pair of numbers $1\leq b_j \leq d_j\leq m$, denoted $[b_j,d_j]_j$. For $1\leq s\leq t\leq n$, a \emph{staircase} from $s$ to $t$ is a set of slices $[b_j,d_j]_j$ for $s \leq j \leq t$ such that $b_{j+1}\leq b_{j}\leq d_{j+1}\leq d_{j}$ for any $j\in \{s,\dots,t-1\}$. To make explicit the constants $m$ and $n$, we say that such a set of slices is a staircase of $\Gf{m,n}$.

\begin{proposition} \label{prp:stair}
  Let $\II{m,n}'$ be the set of all staircases of $\Gf{m,n}$.  There exists a bijection between $\II{m,n}$ and $\II{m,n}'$.
\end{proposition}

\begin{proof}
  We construct a set bijection $f:\II{m,n} \longrightarrow \II{m,n}'$ together with its inverse $f^{-1}$.

  For each interval subquiver $I \in \II{m,n}$, we define $f(I)$ to be the set of slices $f(I) \doteq \{[b_j,d_j]_j\suchthat s\leq j\leq t\}$ from $s$ to $t$, where
    \[
      s = \min \{ j \suchthat (i,j) \in I \text{ for some }i\} \text{ and }
      t = \max \{ j \suchthat (i,j) \in I \text{ for some }i\}
    \]
    and for $s\leq j \leq t$,
    \[
      b_j = \min\{i \suchthat (i,j) \in I\} \text{ and }
      d_j = \max\{i \suchthat (i,j) \in I\}.
    \]
    Note that since $I$ is nonempty, $1\leq s \leq t \leq n$. Then, for each $j$ with $s\leq j \leq t$, the set $\{i \suchthat (i,j) \in I\}$ is nonempty by the connectedness condition, 
    and thus $1\leq b_j \leq  d_j \leq m$.
    Similarly, $b_{j}\leq d_{j+1}$ follows from the connectedness of $I$.

    The correctness of conditions $b_{j+1}\leq b_{j}$ and $d_{j+1}\leq d_{j}$ follows from the convexity of $I$. To see this, suppose to the contrary that $b_{j+1} > b_{j}$. Then, we have a path $(b_j,j)$ to $(b_j,j+1)$ to $(b_{j+1},j+1)$ with both endpoints in $I$, but $(b_j,j+1)$ is not in $I$ since $b_j < b_{j+1} = \min\{i \suchthat (i,j+1) \in I\}$. This contradicts convexity. A similar argument shows that $d_{j+1}\leq d_{j}$.
    The above arguments show that $f(I)$ is indeed a staircase.

    In the opposite direction, given a staircase $I' \doteq \{[b_j,d_j]_j\suchthat s\leq j\leq t\}$ from $s$ to $t$, we define $f^{-1}(I')$ to be the full subquiver with vertices
    \[
      \{(i,j) \suchthat s\leq j \leq t, b_j\leq i\leq d_j\}.
    \]
    It is clear that $f$ and $f^{-1}$ are inverses of each other.
\end{proof}

In general, for a representation $V\in \rep_K (Q, R)$ with $\# Q_0 =n$, the \emph{dimension vector} of $V$ is defined to be
\[
\underline{\dim} V := (\dim_K V(x))_{x \in Q_0} \in \bbZ^n .
\]
When we display dimension vectors, we position the numbers $\dim_K V(x)$ corresponding to the position where each vertex $x \in Q_0$ is graphically displayed (see Example~\ref{ex:correspondence}). By definition, each interval representation $M$ of $\Gf{m,n}$ can be uniquely expressed by its dimension vector, since it is uniquely determined by its support.

By Proposition~\ref{prp:stair}, we identify interval subquivers of $\Gf{m,n}$ with staircases of $\Gf{m,n}$. Thus, we shall also denote an interval by writing it as a set of slices $\{[b_j,d_j]_j\suchthat s \leq j \leq t\}$, as a staircase from $s$ to $t$.
We can visualize the correspondence $f:\II{m,n} \longrightarrow \II{m,n}'$ in the proof of Proposition \ref{prp:stair} using the dimension vector notation and staircase notation. Below, we illustrate some examples under this correspondence for $\Gf{6,4}$.
\begin{example} The following are examples of intervals in $\Gf{6,4}$.
\label{ex:correspondence}
\[
    \begin{pmatrix}
      011100 \\[-4pt] 001100 \\[-4pt] 001110 \\[-4pt] 000011
    \end{pmatrix}
    \longleftrightarrow
    \{ [5,6]_1, [3,5]_2, [3,4]_3, [2,4]_4 \},
    \begin{pmatrix}
      000000 \\[-4pt] 011100 \\[-4pt] 001110 \\[-4pt] 000000
    \end{pmatrix}
    \longleftrightarrow
    \{ [3,5]_2, [2,4]_3 \}.
  \]
\end{example}

Using this staircase shape, we are able to prove the following
\begin{lemma}\label{lem:walk}
Let $m, n$ be positive integers.  Any pre-interval representation of $\Gf{m,n}$ is isomorphic to an interval representation.
\end{lemma}

\begin{proof}
Let $V$ be a pre-interval representation of $\Gf{m,n} = \Af{m}\otimes \Af{n}$.
Then $\supp(V)$ is an interval by definition, and thus a staircase by Proposition \ref{prp:stair}.
Set $B$ to be the quiver $\supp(V)$ with full commutativity relations in it.  Then $V$ is regarded as a representation of $B$.

Let $B'$ be the bound quiver obtained from $\supp(V)$ by 
flipping
all 
of its
vertical arrows, together
with the full commutativity relations.
Thus the quiver of $B'$ is a subquiver of $\Af{m}\otimes (\Af{n})^{\mathrm{op}}$.
Then by replacing all maps of $V$ associated to the vertical arrows in $\supp(V)$ (which are nonzero by definition) by their inverses, we obtain a representation $V'$ of the bound quiver $B'$.
To see that the commutative relations in $B'$ are satisfied by $V'$, we note that the left square below is a commutative diagram of nonzero linear maps if and only if  the right one is:
\[
  \xymatrix{
    K & K\\
    K & K
    \ar"1,1"; "1,2"^{a_1}
    \ar"2,1"; "2,2"_{a_2}
    \ar"2,1"; "1,1"^{b_1}
    \ar"2,2"; "1,2"_{b_2}
  }
  \qquad
  \xymatrix{
    K & K\\
    K & K\mathrlap{,}
    \ar"1,1"; "1,2"^{a_1}
    \ar"2,1"; "2,2"_{a_2}
    \ar"1,1"; "2,1"_{b_1\inv}
    \ar"1,2"; "2,2"^{b_2\inv}
}
\]
because $a_1 b_1 = b_2 a_2$ is equivalent to $b_2\inv a_1 = a_2 b_1\inv$.

We illustrate the construction with the following example, showing the quiver of $B$ and $B'$, respectively:
\begin{equation}
  \label{diag:flipped_vertical}
  \xymatrix{
    \circ&\circ&\circ&\circ\\
    \circ&\circ&\circ&\circ&\circ\\
    &&\circ&\circ&\circ
    \ar"1,1";"1,2"  \ar"1,2";"1,3"  \ar"1,3";"1,4"
    \ar"2,1";"2,2"  \ar"2,2";"2,3"  \ar"2,3";"2,4" \ar"2,4";"2,5"
    \ar"3,3";"3,4"  \ar"3,4";"3,5"
    \ar"2,1";"1,1"  \ar"2,2";"1,2"  \ar"2,3";"1,3" \ar"2,4";"1,4"
    \ar"3,3";"2,3"  \ar"3,4";"2,4"  \ar"3,5";"2,5"
  }
  \qquad  
  \xymatrix{
    x&\circ&\circ&\circ\\
    \circ&\circ&\circ&\circ&\circ\\
    &&\circ&\circ&\circ\mathrlap{.}
    \ar"1,1";"1,2"  \ar"1,2";"1,3"  \ar"1,3";"1,4"
    \ar"2,1";"2,2"  \ar"2,2";"2,3"  \ar"2,3";"2,4" \ar"2,4";"2,5"
    \ar"3,3";"3,4"  \ar"3,4";"3,5"
    \ar"1,1";"2,1"  \ar"1,2";"2,2"  \ar"1,3";"2,3" \ar"1,4";"2,4"
    \ar"2,3";"3,3"  \ar"2,4";"3,4"  \ar"2,5";"3,5"
  }
\end{equation}
We view $V'$ as a representation of $B'$ and not of $\Af{m}\otimes (\Af{n})^{\mathrm{op}}$. 
So for example, there is no problem with the upper right portion of the quiver of $B'$ in Diagram~\eqref{diag:flipped_vertical} not satisfying a zero relation.

In general, let $x$ be the upper left corner of the quiver of $B'$, and take a nonzero element $v_x$ of $K = V'(x)$.
For each vertex $y$ of $B'$ there exists a path $\mu$ from $x$ to $y$ in the quiver of $B'$
, because $\supp(V)$ has a staircase shape. Take $v_y:= V'(\mu)v_x$ as the basis of $V'(y)$.
Since $B'$ 
is defined by the full commutativity relations, $v_y$ does not depend on the choice of $\mu$.
In this way we can find bases $v_y$ of $V'(y)$ for all vertices $y$ in $\supp (V')$
that satisfy Condition~\eqref{eq:iso-interval} in Remark~\ref{rem:interval}.  Now $v_y$ are also bases of $V(y) = V'(y)$ for all
$y \in \supp(V)_0$ and satisfy Condition \eqref{eq:iso-interval} for $V$.
Thus $V$ is isomorphic to an interval representation.
\end{proof}

Finally, we prove the main result of this subsection.
\begin{theorem}
  \label{thm:thin_interval}
  Let $m, n$ be positive integers. Let $M$ be a thin indecomposable representation of the commutative grid $\Gf{m,n}= (Q, R)$. Then $M$ is isomorphic to an interval representation.
\end{theorem}

\begin{proof}
    The proof will be done by contradiction and in two steps.
    First we show that any thin indecomposable representation that is not a pre-interval should have two non-zeros vector spaces with a path containing a zero map between them.
    Then we will show that this implies that the representation is decomposable.
    Lemma~\ref{lem:walk} will then allow us to conclude.

    Assume by contradiction that $M$ is a thin indecomposable that is not a pre-interval representation. As $M$ is an indecomposable representation, its support is connected.
    Therefore, either the convexity condition on the support of $M$ fails, or the nonzero maps on support condition fails. In the first case, there exist vertices $x, y, z \in Q_0$ such that there is a path from $x$ to $y$ to $z$, and $M(x)\neq 0$, $M(y)=0$ and $M(z)\neq 0$. In the second case, there exists an arrow $\alpha:x\rightarrow z$ in $Q_1$ with $M(x)\neq 0$, $M(z)\neq 0$ and $M(\alpha)=0$.

    In either case, we have a path $p$ from $x$ to $z$ with $M(x)\neq 0$ and $M(z)\neq 0$ such that $p$ contains an arrow $\alpha$ with $M(\alpha)= 0$.

    Let us consider the representation $M$ on a square with one corner at $(i,j)\in Q_0$ in the grid:
    \begin{equation}
      \begin{tikzcd}
        M(i,j+1) \rar{t} & M(i+1,j+1) \\
        M(i,j) \uar{l} \rar{b} & M(i+1,j) \uar{r}
      \end{tikzcd}
    \end{equation}
    where the maps are the values of the representation $M$ on the arrows (for example, $r \doteq M(\beta)$ where $\beta$ is the arrow $\beta:(i+1,j)\rightarrow (i+1,j+1)$).
    By full commutativity, the two paths (compositions of maps) from $M(i,j)$ to $M(i+1,j+1)$ are equal: $rb = tl$. Since the vector spaces have dimension at most $1$ as $M$ is thin, we can conclude the following.
    If at least one map is zero on one of these paths, then there is a zero map on the other path.


    We use the above observation to build a line $L$ intersecting only zero maps in $M$ across the grid that separates $M$. We start with the arrow $\alpha$ with $M(\alpha) = 0$ found previously and inductively build this line using the following observation. At each square of the grid, at least one of the following patterns is possible:
    \begin{equation}
        \begin{tikzpicture}[commutative diagrams/every diagram,scale=0.75]
            \node (A) at (0,0) {$M(a)$};
            \node (B) at (2,0) {$M(b)$};
            \node (C) at (0,2) {$M(c)$};
            \node (D) at (2,2) {$M(d)$};
            \draw[->] (A) -- (B);
            \draw[->] (B) -- (D);
            \draw[->] (A) -- (C);
            \draw[->] (C) -- (D);
            \draw[red] (-.5,1) -- (1,1) -- (1,-.5);
        \end{tikzpicture}
        \hspace{1em}
        \begin{tikzpicture}[commutative diagrams/every diagram,scale=0.75]
            \node (A) at (0,0) {$M(a)$};
            \node (B) at (2,0) {$M(b)$};
            \node (C) at (0,2) {$M(c)$};
            \node (D) at (2,2) {$M(d)$};
            \draw[->] (A) -- (B);
            \draw[->] (B) -- (D);
            \draw[->] (A) -- (C);
            \draw[->] (C) -- (D);
            \draw[red] (-.5,1) -- (1,1) -- (2.5,1);
        \end{tikzpicture}
        \hspace{1em}
        \begin{tikzpicture}[commutative diagrams/every diagram,scale=0.75]
            \node (A) at (0,0) {$M(a)$};
            \node (B) at (2,0) {$M(b)$};
            \node (C) at (0,2) {$M(c)$};
            \node (D) at (2,2) {$M(d)$};
            \draw[->] (A) -- (B);
            \draw[->] (B) -- (D);
            \draw[->] (A) -- (C);
            \draw[->] (C) -- (D);
            \draw[red] (1,2.5) -- (1,1) -- (1,-.5);
        \end{tikzpicture}
        \hspace{1em}
        \begin{tikzpicture}[commutative diagrams/every diagram,scale=0.75]
            \node (A) at (0,0) {$M(a)$};
            \node (B) at (2,0) {$M(b)$};
            \node (C) at (0,2) {$M(c)$};
            \node (D) at (2,2) {$M(d)$};
            \draw[->] (A) -- (B);
            \draw[->] (B) -- (D);
            \draw[->] (A) -- (C);
            \draw[->] (C) -- (D);
            \draw[red] (1,2.5) -- (1,1) -- (2.5,1);
        \end{tikzpicture}
    \end{equation}
    where in each pattern, the line (colored red) intersects a pair of arrows $\beta_1,\beta_2$ where $M(\beta_1)=0$ and $M(\beta_2)=0$. Note that if more maps are zero, we simply ignore them and choose to extend our line using only one of the four given patterns.

    As we are working over a finite $2$D grid, this line cannot create a circle. Therefore it goes from one boundary of the grid to another, and divides the grid into two regions with vertices we denote by $V_\ell$ and $V_r$, for ``left/bottom'' and ``right/top'', respectively. Furthermore, both regions are non-trivial: by construction, $x\in V_\ell$ and $z\in V_r$ with $M(x)\neq 0$ and $M(z) \neq 0$ since the arrow $\alpha$ was found as part of a path from vertex $x$ to $z$ with those properties.




    Let $Q_\ell = (V_\ell,E(V_\ell))$ and $Q_r = (V_r, E(V_r))$ be the full subquivers generated by $V_\ell$ and $V_r$ respectively, and let $E(L)$ be the set of the arrows intersecting the line $L$ constructed above. Then, the grid is partitioned as $\Gf{m,n} = (Q_0,Q_1) = (V_\ell \sqcup V_r, E(V_\ell) \sqcup E(L) \sqcup E(V_r))$. To see this, we note that by construction $E(V_\ell)$ and $E(V_r)$ are disjoint. Furthermore, $E(L)$ is by definition the arrows going from a vertex of $V_\ell$ to $V_r$, and is disjoint from $E(V_\ell)$ and $E(V_r)$. Finally, each arrow on the grid is in one of these three sets. In Figure~\ref{fig:proof_picture}, we illustrate this partitioning.

    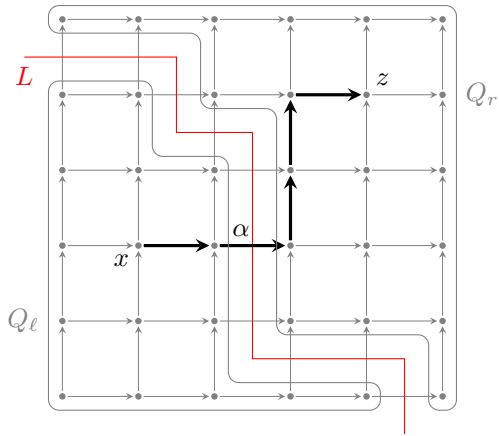
\begin{figure}[h!]
      \begin{center}
      \begin{tikzpicture}[nodestyle/.style={circle, draw, color=gray, fill=gray, minimum size=2, inner sep=0,outer sep=1pt}]
        \foreach \x in {0,...,5} {
          \foreach \y in {0,...,5} {
            \node[nodestyle] (\x\y) at (\x,\y) {};
          }
        }

        \foreach \x in {0,...,5} {
          \foreach \y [count=\yi] in {0,...,4} {
            \draw[very thin, color=gray, -stealth] (\x\y) edge (\x\yi) (\y\x) edge (\yi\x);
          }
        }

        \draw[red] (-0.5,4.5) -- (1.5,4.5) -- (1.5,3.5) -- (2.5,3.5) -- (2.5, 0.5) -- (4.5,0.5)--(4.5,-0.5);
        \draw[black, very thick, -stealth]
        (12) edge (22)
        (22) edge (32)
        (32) edge (33)
        (33) edge (34)
        (34) edge (44);

        \node[anchor=south] at (2.35,2) {$\alpha$};
        \node[anchor=north east] at (12) {$x$};
        \node[anchor=south west] at (44) {$z$};
        \node[anchor=north,red] at (-0.5,4.5) {$L$};

        \foreach \x in {0,...,5} {
          \foreach \y in {0,...,5} {
            \node[fill=none,draw=none,inner sep=5pt] (\x\y) at (\x,\y) {};
          }
        }
        \draw[rounded corners, draw=gray] (00.south west)--(04.north west)--(14.north east)--(13.north east)--(23.north east)-- (20.north east)-- (40.north east)--(40.south east)--cycle;
        \node[anchor=east,gray] at (01.west) {$Q_\ell$};

        \draw[rounded corners, draw=gray] (05.north west)--(55.north east) -- (50.south east)--(50.south west)--(51.south west)--(31.south west)--(34.south west)--(24.south west) --(25.south west)--(05.south west)--cycle;
        \node[anchor=west,gray] at (54.east) {$Q_r$};
      \end{tikzpicture}
      \end{center}
      \caption{Starting with the detected path $p$ (thick line) from $x$ to $z$ with $M(x)$ and $M(z)$ nonzero, we find an arrow $\alpha$ with $M(\alpha)=0$. This region of zeros propagates to the red line $L$, which divides $M$ into two.}
      \label{fig:proof_picture}
    \end{figure}

    Consider representations $M_\ell$ and $M_r$ obtained by setting $M$ to be zero outside of $Q_\ell$ and $Q_r$ respectively. The support of $M_\ell$ is included in $Q_\ell$. Note that by construction the arrows exiting $Q_\ell$ are exactly the arrows $E(L)$, which all support a zero map in $M$.  Hence $M_\ell$ is a subrepresentation of $M$. Clearly,  $M_r$ is a subrepresentation of $M$ since there are no arrows exiting $Q_r$. Furthermore, as $M$ restricted to $E(L)$ is $0$, we conclude that $M = M_\ell\oplus M_r$.


    By the fact that $M_l(x)\neq 0$ and $M_r(z)\neq 0$, it follows that the decomposition above is nontrivial, and thus $M$ is decomposable, a contradiction.  Therefore $M$ is a pre-interval representation, and Lemma~\ref{lem:walk} implies that $M$ is isomorphic to an interval representation.
\end{proof}


\subsection{Interesting examples}
\label{subsec:examples}

In this subsection, we give some interesting examples of where a thin indecomposable may not be isomorphic to an interval representation. 


Over the equioriented commutative $3$D grid, we provide the following example. Let $\lambda$ be any element of $K$, and define
\[
M(\lambda):
  \begin{tikzcd}[row sep=1.2em]
    & K \ar{rr}{} & & 0
    \\
    K \ar{ur}{1} \ar{rr}[near end]{1} & & K \ar{ur}{} &
    \\
    & K \ar{rr}[near start]{1} \ar{uu}[near start]{1} & & K \ar{uu}{}
    \\
    0 \ar{uu}{} \ar{ur}{} \ar{rr}{} & & K\ar{uu}[near start]{1} \ar{ur}[swap]{\lambda} &
  \end{tikzcd}
\]
This, and higher-dimensional versions of this indecomposable were studied in the paper \cite{buchet_et_al:socg}, where topological realizations were also given for $\lambda=0$. It is easy to see that $M(\lambda)$ is indecomposable, and for any $\lambda \neq \mu \in K$, $M(\lambda) \not\cong M(\mu)$. Furthermore, $M(\lambda)$ is thin. However, if $\lambda \neq 1$, $M(\lambda)$ is not an interval representation, and is not isomorphic to one. Moreover $M(0)$ is not a pre-interval representation and is not isomorphic to one but is still a thin indecomposable.

Next, if the arrows are not oriented in the same direction, some thin indecomposables may not be interval representations. An example is the representation
\begin{equation}
  \label{diag:counter_nonequi}
  \begin{tikzcd}
    K & K \ar{l}[swap]{1} \ar{r}{1} & K\\
    K \ar{u}{1} \ar{d}[swap]{1} & 0 \ar{l} \ar{r} \ar{u} \ar{d} & K \ar{u}[swap]{1} \ar{d}{1} \\
    K & K \ar{l}{1} \ar{r}[swap]{\lambda} & K
  \end{tikzcd}
\end{equation}
of a non-equioriented commutative $2$D grid, where $\lambda$ is not $1$. If $\lambda$ is not $0$ and not $1$, this also gives an example of a pre-interval representation that is not an interval representation (and not isomorphic to one).

The above are variations on the same theme: we have an example of a thin indecomposable that is not pre-interval representation (when $\lambda = 0$), and an example of a pre-interval representation that is not isomorphic to an interval representation (when $\lambda$ is not $0$ nor $1$).
Hence we have strict inclusions in the hierarchy of Ineq.~\eqref{eq:hierarchy}.

Next, let us provide an example of a bound quiver $(Q,R)$ where pre-interval representations are always isomorphic to an interval representation, but thin indecomposables are not always pre-interval representations. Consider the quiver
\[
  \begin{tikzcd}
    \bullet \rar[bend right=30,swap]{x} & \bullet \lar[bend right=30,swap]{y}
  \end{tikzcd}
\]
with relations $R=\{xy-xyxy, yx-yxyx\}$. Then,
\[
  \begin{tikzcd}
    K \rar[bend right=30,swap]{0} & K \lar[bend right=30,swap]{1}
  \end{tikzcd}
\]
is an example of a thin indecomposable representation of $(Q,R)$ that is not a pre-interval representation.

Now suppose that $V\in \rep_K(Q,R)$ is a pre-interval representation. In the case that $V$ is a simple representation, it is automatically an interval representation. Otherwise, $V$ is isomorphic to some
\[  
  \begin{tikzcd}
    K \rar[bend right=30,swap]{f} & K \lar[bend right=30,swap]{g}
  \end{tikzcd}
\]
Then, the relations $R$ imply that $fg-fgfg=0$ and $gf-gfgf=0$. Together with the fact that $f$ and $g$ are nonzero because $V$ is a pre-interval representation, we see that $f$ and $g$ are mutually-inverse isomorphisms. Thus, $V$ is isomorphic to
\[  
  \begin{tikzcd}
    K \rar[bend right=30,swap]{1} & K \lar[bend right=30,swap]{1}
  \end{tikzcd}
\]
which is an interval representation.





\subsection{Listing all 2D intervals} \label{sub:listing}
By definition, an interval representation can be uniquely identified with its support, an interval subquiver.
Recall that $\II{m,n}$ is the set of all nonempty interval subquivers of the equioriented commutative $2$D grid $\Gf{m,n}$. In this subsection, we count the elements of $\II{m,n}$.
Recall that by Proposition~\ref{prp:stair}, we identify interval subquivers of $\Gf{m,n}$ with staircases of $\Gf{m,n}$, and denote an interval by writing it as a set of slices $\{[b_j,d_j]_j\suchthat s \leq j \leq t\}$, as a staircase from $s$ to $t$.


\begin{definition}[Size of interval]
  For an interval $I = \{ [b_j,d_j]_j \suchthat s\leq j \leq t\}$ (a staircase from $s$ to $t$), we define the size of I as follows.
  \[
    \size(I) \doteq (d_t-b_s+1,t-s+1)\in \mathbb{Z}^2
  \]
  Moreover, for each $(w,h)\in \{1,\dots,m\}\times\{1,\dots,n\}$, we set 
  \begin{align*}
    F_{m,n}(w,h)& \doteq \{ I \in \II{m,n} \suchthat \size(I) = (w,h) \} \text{ and}\\
    R(w,h)& \doteq \{ I \in \II{w,h} \suchthat \size(I) = (w,h) \} = F_{w,h}(w,h)
  \end{align*}
\end{definition}
While both sets contains staircases of the same size $(w,h)$, the underlying 2D commutative grid is different. The set $F_{m,n}(w,h)$ considers staircases in $\Gf{m,n}$, but $R(w,h)$ considers only staircases from $\Gf{w,h}$ that are of maximum size.

\begin{example}
If $(m,n)=(3,3)$, then 
\begin{align*}
F_{3,3}(2,2)&=\left\{\rule{-4pt}{10pt}\right.
\left(\rule{-2pt}{10pt}\right.
\begin{matrix} 110 \\[-4pt] 110 \\[-4pt] 000 \end{matrix}
\left.\rule{-2pt}{10pt}\right),
\left(\rule{-2pt}{10pt}\right.
\begin{matrix} 011 \\[-4pt] 011 \\[-4pt] 000 \end{matrix}
\left.\rule{-2pt}{10pt}\right),
\left(\rule{-2pt}{10pt}\right.
\begin{matrix} 000 \\[-4pt] 110 \\[-4pt] 110 \end{matrix}
\left.\rule{-2pt}{10pt}\right),
\left(\rule{-2pt}{10pt}\right.
\begin{matrix} 000 \\[-4pt] 011 \\[-4pt] 011 \end{matrix}
\left.\rule{-2pt}{10pt}\right),
\left(\rule{-2pt}{10pt}\right.
\begin{matrix} 100 \\[-4pt] 110 \\[-4pt] 000 \end{matrix}
\left.\rule{-2pt}{10pt}\right),
\left(\rule{-2pt}{10pt}\right.
\begin{matrix} 010 \\[-4pt] 011 \\[-4pt] 000 \end{matrix}
\left.\rule{-2pt}{10pt}\right),
\left(\rule{-2pt}{10pt}\right.
\begin{matrix} 000 \\[-4pt] 100 \\[-4pt] 110 \end{matrix}
\left.\rule{-2pt}{10pt}\right),
\left(\rule{-2pt}{10pt}\right.
\begin{matrix} 000 \\[-4pt] 010 \\[-4pt] 011 \end{matrix}
\left.\rule{-2pt}{10pt}\right),\\
&\left(\rule{-2pt}{10pt}\right.
\begin{matrix} 110 \\[-4pt] 010 \\[-4pt] 000 \end{matrix}
\left.\rule{-2pt}{10pt}\right),
\left(\rule{-2pt}{10pt}\right.
\begin{matrix} 011 \\[-4pt] 001 \\[-4pt] 000 \end{matrix}
\left.\rule{-2pt}{10pt}\right),
\left(\rule{-2pt}{10pt}\right.
\begin{matrix} 000 \\[-4pt] 110 \\[-4pt] 010 \end{matrix}
\left.\rule{-2pt}{10pt}\right),
\left(\rule{-2pt}{10pt}\right.
\begin{matrix} 000 \\[-4pt] 011 \\[-4pt] 001 \end{matrix}
\left.\rule{-2pt}{10pt}\right)
\left.\rule{-4pt}{10pt}\right\}
\text{and }R(2,2)=\{ 
(\begin{matrix} 11 \\[-4pt] 11 \end{matrix}),
(\begin{matrix} 10 \\[-4pt] 11 \end{matrix}),
(\begin{matrix} 11 \\[-4pt] 01 \end{matrix})
\}.
\end{align*}
\end{example}
It is clear that the elements of $F_{m,n}(w,h)$ can be listed by shifting each element of $R(w,h)$, and thus
\[
  \card{F_{m,n}(w,h)} = (m-w+1)(n-h+1) \card{R(w,h)}.
\]
Moreover, the cardinality of the set of all intervals $\II{m,n}$ can thus be obtained by counting intervals of all possible sizes:
  
\begin{equation} \label{eq:I}
\card{\II{m,n}} = \sum\limits_{w=1}^m \sum\limits_{h=1}^n \card{F_{m,n}(w,h)}
  =\sum\limits_{w=1}^m \sum\limits_{h=1}^n (m-w+1)(n-h+1)\card{R(w,h)}.
\end{equation}

Thus, to calculate $\card{\II{m,n}}$, it is enough to calculate the numbers $\card{R(w,h)}$.
Next we give an explicit form for the value of $\card{R(w,h)}$, by relating it to a well-known concept in combinatorics.
\begin{definition}[Parallelogram polyomino]
A parallelogram polyomino having a $w\times h$ bounding box is a polyomino contained in a rectangle consisting of $wh$ cells 
and formed by cutting out, from this rectangle, two (possibly empty) non-touching
Young diagrams with corners at $(0,0)$ and $(w,h)$. 
\end{definition}
Equivalently, a parallelogram polyomino with a $w\times h$ bounding box is a pair of lattice non-increasing paths $P$, $Q$ from $(0,h)$ to $(w,0)$ so that $P$ lies entirely above $Q$, and $P$ and $Q$ intersect only at $(0,h)$ and $(w,0)$. In the definition above, $h$ is taken to be the height and $w$ the width. An example of a parallelogram polyomino having a $6\times 4$ bounding box is given below.
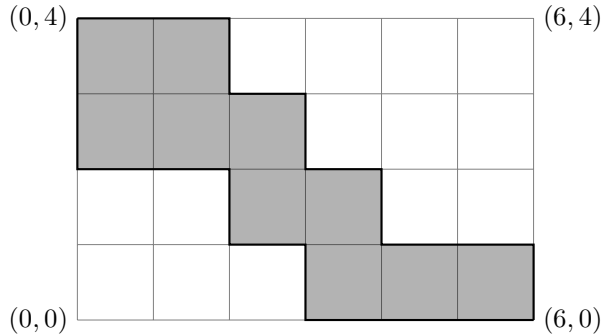
\[
  \begin{tikzpicture}
    \draw[draw=gray,ultra thin] (0,0) grid (6,4);
    \filldraw[fill=black,opacity=.3] (0,4)--(2,4)--(2,3)--(3,3)--(3,2)--(4,2)--(4,1)--(6,1)--(6,0)--(3,0)--(3,1)--(2,1)--(2,2)--(0,2)--cycle;

    \draw[black,thick] (0,4)--(2,4)--(2,3)--(3,3)--(3,2)--(4,2)--(4,1)--(6,1)--(6,0);
    \draw[black,thick] (6,0)--(3,0)--(3,1)--(2,1)--(2,2)--(0,2)--(0,4);
    \node[anchor=east] at (0,0) {$(0,0)$};
    \node[anchor=west] at (6,4) {$(6,4)$};
    \node[anchor=east] at (0,4) {$(0,4)$};
    \node[anchor=west] at (6,0) {$(6,0)$};
  \end{tikzpicture}
\]

By interpreting staircases $I \in \II{w,h}$ as the filled-in boxes on the lattice (not the grid lines!), it is clear that staircases in $R(w,h)$ are in one to one correspondence with parallelogram polyominoes with a $w \times h$  bounding box. The example above is identified to a staircase in the following way.
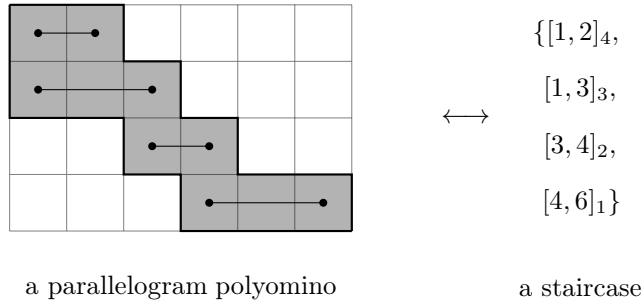
\[
  \begin{tikzpicture}[scale=0.75, baseline=(current bounding box.center)]
    \draw[draw=gray,ultra thin] (0,0) grid (6,4);
    \filldraw[fill=black,opacity=.3] (0,4)--(2,4)--(2,3)--(3,3)--(3,2)--(4,2)--(4,1)--(6,1)--(6,0)--(3,0)--(3,1)--(2,1)--(2,2)--(0,2)--cycle;

    \draw[black,thick] (0,4)--(2,4)--(2,3)--(3,3)--(3,2)--(4,2)--(4,1)--(6,1)--(6,0);
    \draw[black,thick] (6,0)--(3,0)--(3,1)--(2,1)--(2,2)--(0,2)--(0,4);

    \tikzset{dotstyle/.style={circle,fill=black,inner sep=0pt,minimum size=3pt}} 
    \draw (0.5,3.5) node[dotstyle]{} -- (1.5,3.5) node[dotstyle]{};
    \draw (0.5,2.5) node[dotstyle]{} -- (2.5,2.5) node[dotstyle]{};
    \draw (2.5,1.5) node[dotstyle]{} -- (3.5,1.5) node[dotstyle]{};
    \draw (3.5,0.5) node[dotstyle]{} -- (5.5,0.5) node[dotstyle]{};

    \node at (8,2) {$\longleftrightarrow$};
    
    \node at (10,3.5) {$\left\{[1,2]_4,\right.$};
    \node at (10,2.5) {$[1,3]_3,$};
    \node at (10,1.5) {$[3,4]_2,$};
    \node at (10,0.5) {$\left.[4,6]_1\right\}$};

    \node at (3,-1) {a parallelogram polyomino};
    \node at (10,-1) {a staircase};
  \end{tikzpicture}  
\]
\begin{lemma} \label{lem:bij}
  There is a bijection between $R(w,h)$ and
  the set of parallelogram polyominoes with $w\times h$ bounding box. 
\end{lemma}

\begin{definition}[Narayana number]
  For each pair of integers $1\leq b\leq a$, the Narayana number $N(a,b)$ is defined by using binomial coefficients as follows.
  $$
  N(a,b):=\frac{1}{a}
  \begin{pmatrix} a \\ b \end{pmatrix}
  \begin{pmatrix} a \\ b-1 \end{pmatrix}
  $$
\end{definition}
Narayana numbers are closely related with
counting problems of parallelogram polyominoes.
Especially, the following fact is well known. For example, see  \cite{aval2014statistics}.
\begin{proposition} \label{prp:PP}
  The number of parallelogram polyominoes having an $w\times h$ bounding box is exactly $N(h+w-1,h)$.
\end{proposition}

Hence we obtain the following formulas.

\begin{theorem}
  \label{thm:num-Imn}
  Let $m$, $n$ be positive integers. Then:
  \[
    \card{R(w,h)} = N(h+w-1,h)=\displaystyle\frac{1}{h+w-1}
    \begin{pmatrix} h+w-1 \\ h-1 \end{pmatrix}
    \begin{pmatrix} h+w-1 \\ w-1 \end{pmatrix}
  \]
  and  
  \[
    \card{\II{m,n}}=\sum_{w=1}^{m}\sum_{h=1}^{n}\frac{(m-w+1)(n-h+1)}{(h+w-1)}
    \begin{pmatrix} h+w-1 \\ h-1 \end{pmatrix}
    \begin{pmatrix} h+w-1 \\ w-1 \end{pmatrix}.
  \]
\end{theorem}
\begin{proof}
  We use Lemma~\ref{lem:bij} and Proposition~\ref{prp:PP} and note that
  \[
    \begin{pmatrix} h+w-1 \\ h \end{pmatrix} =
    \begin{pmatrix} h+w-1 \\ w-1 \end{pmatrix}
  \]
  to obtain the first formula, and second formula follows from Equation~\eqref{eq:I}.
\end{proof}

In particular, for an equioriented commutative $2$D grid of size $m\times 2$ (an equioriented commutative ladder \cite{escolar2016persistence}), we obtain the following formula.

\begin{corollary}
For each $ m \in \bbN$, we have
$$
\#I_{m,2} = \frac{1}{24}m(m+1)(m^2+5m+30).
$$
\end{corollary}

\begin{remark}
We can apply Theorem~\ref{thm:main_oracle} to a given representation $M$ of $\Gf{m,n}$ in order to determine whether or not it is interval-decomposable. Then, Theorem~\ref{thm:num-Imn} gives the cardinality of the set $\calS$ of intervals over which we need to compute multiplicities. This cardinality is a large number. To mitigate this, we may replace the original quiver $\Gf{m,n}$ by the smallest equioriented commutative 2D-grid containing the support of $M$.
\end{remark}

\section{Algorithms and computational complexity}
\label{sec:cc}

  In this section, we provide a detailed algorithm for determining interval-decomposability, based on Theorem~\ref{thm:main_oracle}.
  In the final subsection, we also give a remark concerning the use of the decomposition algorithm given in \cite{dey2019generalized}, for computing interval-decomposability.
Here, we let $\omega < 2.373$ be the matrix multiplication exponent \cite{don1990matrix,williams2012multiplying}.

Given a $2$D persistence module $M$ in $Q=\Gf{m,n}$, the following procedure can be used to determine whether or not $M$ is interval-decomposable.

\begin{algorithm}[H]
  \caption{Algorithm for determining interval-decomposability of $M$}
  \label{alg:overall}
  \begin{algorithmic}[1]
    \Function{isIntervalDecomposable}{$M$}
    \State dimVecRemaining $\gets \dimv M$
    \For {$x = m, m-1, \dots , 1$}
      \For {$y = 1, 2, \dots , n$}
        \If{$\mathrm{dimVecRemaining}_{x, y}=0$}
          \State ${\bf continue}$
        \EndIf
        \For {$L \in \Call{getCandidates}{x, y}$}
          \State $d_M(L) \gets $ \Call{multiplicity}{$M$, $L$}
          \State {dimVecRemaining $\gets$ dimVecRemaining - $d_M(L) \dimv L$}\label{alg:overall_decrement}
          \If{$\mathrm{dimVecRemaining}_{x, y}=0$} \State ${\bf break}$ \EndIf
        \EndFor
        \If{$\mathrm{dimVecRemaining}_{x, y}>0$}
          \State \Return
        False \EndIf
      \EndFor
    \EndFor
    \State \Return True
    \EndFunction
  \end{algorithmic}
\end{algorithm}

Let us first give an overview of Algorithm~\ref{alg:overall}. We initialize $\mathrm{dimVecRemaining}$, which holds the dimensions of vector spaces yet unprocessed by the algorithm. In particular, we let  $\mathrm{dimVecRemaining}_{x,y}$ hold the dimension at $(x,y)$ i.e. column $x$ and row $y$ counting from the bottom. For example, below is the underlying quiver of $\Gf{4,3}$ which has $4$ columns and $3$ rows. For clarity, the $(x,y)$ coordinates of the corner points are labelled.
\[
  \begin{tikzcd}[row sep=1em, column sep=1em]
    \mathllap{(1,3)}\bullet \rar & \bullet \rar & \bullet \rar & \bullet\mathrlap{(4,3)} \\
    \bullet \rar\uar & \bullet \rar\uar & \bullet\uar\rar & \bullet\uar \\
    \mathllap{(1,1)}\bullet \rar\uar & \bullet \rar\uar & \bullet\uar\rar & \bullet\mathrlap{(4,1)}\uar
  \end{tikzcd}
\]

The main action happens in Line~\ref{alg:overall_decrement}, where we decrement $\mathrm{dimVecRemaining}$ by the dimension vector of some interval $L$ multiplied by its multiplicity $d_M(L)$ in $M$.
Ignoring for a moment all the places where the algorithm can terminate early,
if we simply iterate through all intervals $L$ of the grid $\Gf{m,n}$, then by Theorem~\ref{thm:main_oracle} $M$ is interval-decomposable if and only if $\mathrm{dimVecRemaining}_{x,y}$ is $0$ for all $(x,y)$, at the end of the algorithm.

Algorithm~\ref{alg:overall} orders the processing of the intervals $L$ so that there is a possibility of stopping early.
In particular, we order the intervals by their lower-right corners $(x,y)$, in order of decreasing $x$ and increasing $y$ (the two outer for-loops in Algorithm~\ref{alg:overall}).
The procedure \textproc{getCandidates}$(x,y)$ (in Algorithm~\ref{alg:candidates_pointed}), generates the intervals with lower-right corner given by $(x,y)$.
If, after processing all such candidates for some fixed lower-right corner $(x,y)$,
$\mathrm{dimVecRemaining}_{x,y}$ is nonzero,
then we know that $M$ cannot be interval-decomposable.
Indeed, the way we iterate over all possible lower-right corners ensures that once we finish processing $(x,y)$, the value of $\mathrm{dimVecRemaining}_{x,y}$ can no longer change.

\begin{algorithm}[H]
  \caption{Intervals in $Q$ with lower-right corner
    $d_s=x$ and $s=y$}
  \label{alg:candidates_pointed}
  \begin{algorithmic}[1]
    \Function{getCandidates}{$x, y$}
    \State Set an empty list $\mathcal{L}$.
    \For {$b=1, \hdots, x$}
      \State Add the interval $\{\,[b, x]_y\}$
      of height $1$ to the end of list $\mathcal{L}$.
    \EndFor
    \For {$k=0, \hdots, \mathrm{length}(\mathcal{L})-1$}
    \State Read the interval $\mathcal{L}[k]$:
    \StatexIndent[3] (assume that it is the interval from $s$ to $t$ given by $\{\,[b_j, d_j]_j\suchthat s\leq j\leq t\}$)
      \If{$t$ is equal to the height of $Q$} \State ${\bf continue}$ \EndIf
      \For { $b_{t+1}=1, \dots , b_t$ 
      and 
      $d_{t+1}=b_t, \dots, d_t$}
      \State Add the interval 
      $\{ [b_j, d_j]_j \suchthat s\leq j\leq t+1\}$
      to the end of list $\mathcal{L}$.
      \EndFor
    \EndFor
    \State \Return {$\mathcal{L}$}
    \EndFunction
  \end{algorithmic}
\end{algorithm}

Next, let us explain the details of \textproc{getCandidates}$(x,y)$ as presented in Algorithm~\ref{alg:candidates_pointed}.
Recall that we use the following notation for a staircase (an interval, by Proposition~\ref{prp:stair}).
For $1\leq j \leq n$, a \emph{slice} at row $j$ is a pair of numbers $1\leq b_j \leq d_j\leq m$, denoted $[b_j,d_j]_j$. For $1\leq s\leq t\leq n$, a \emph{staircase} from $s$ to $t$ is a set of slices $[b_j,d_j]_j$ for $s \leq j \leq t$ such that $b_{j+1}\leq b_{j}\leq d_{j+1}\leq d_{j}$ for any $j\in \{s,\dots,t-1\}$.
In Algorithm~\ref{alg:candidates_pointed}, we enumerate the candidate intervals with the coordinate of the ``lower-right'' corner fixed as $d_s=x$ and $s=y$. Starting with the lower-right corner, we progressively build up taller and taller intervals.

Next, we also write down Algorithm~\ref{alg:multiplicity} for computing the multiplicity, to be used in Line~\ref{alg:overall_decrement} of Algorithm~\ref{alg:overall}. The correctness of Algorithm~\ref{alg:multiplicity} follows from
formula {\eqref{eq:dual-a}} of
Theorem~\ref{thm:PD-ass}.
The major components of Algorithm~\ref{alg:multiplicity} are the computation of the terms of almost split sequences ending at  nonprojective intervals, which we provide as the function \textproc{almostSplitSequenceTerms} (Algorithm~\ref{alg:ass-p}), and the computation of $\dim\Hom(M,-)$.
\begin{algorithm}[H]
  \caption{Computation of the multiplicity $d_M(L)$ for $L$ an interval}
  \label{alg:multiplicity}
  \begin{algorithmic}[1]
    \Require {$L$ interval}
    \Function{multiplicity}{$M$, $L$}
    \If {$L$ is projective}
      \State $\tau L,\  E_L \gets 0,\ \rad L$
    \Else
      \State $\tau L,\ E_L \gets $ \Call{almostSplitSequenceTerms}{$L$}
    \EndIf
    \State $d_M(L) \gets \dimhom{M}{\tau L} - \dimhom{M}{E_L} + \dimhom{M}{L}$
    \State \Return {$d_M(L)$}
    \EndFunction
  \end{algorithmic}
\end{algorithm}

We devote the next few pages to the discussion of Algorithm~\ref{alg:ass-p}.

\begin{algorithm}[h]
  \caption{Almost split sequence ending at $Z$ nonprojective  with $\End_A(Z)\cong K$}
  \label{alg:ass-p}
  \begin{algorithmic}[1]
    \Require {$Z$ nonprojective indecomposable with $\End_A(Z)\cong K$}
    \Function{almostSplitSequenceTerms}{$Z$}
    \State \label{alg:ass-mpp}
    {Compute a minimal projective presentation:
      $P_1 \xrightarrow{f_1} P_0 \xrightarrow{\varepsilon} Z \to 0$}
    \State \label{alg:ass-nak}
    {Apply the Nakayama functor
      $\nu:= D \circ \Hom_A(\text{-}, A)$ to obtain
      $\nu P_1 \xrightarrow{\nu f_1} \nu P_0$}
    \State \label{alg:ass-trans}
    {Compute $\tau Z:= \Ker(\nu f_1)$}
    \State \label{alg:ass-theta}
    {Compute
      $\theta_Z \colon Z \xrightarrow{\text{\rm can}} \ttop Z  \xrightarrow{\pi} S \hookrightarrow \ttop Z   \xrightarrow{\sim} \ttop P_0 \xrightarrow{\sim} \soc \nu P_0$,}
    \StatexIndent[2] where $S$ is a simple direct summand of $\ttop Z$ 

    \State \label{alg:ass-ans}
    {Compute the middle term $E_Z$ via pullback
      \[
        \begin{tikzcd}[ampersand replacement=\&]
          0 \& X  \&  E_Z  \& Z\\
          0 \& \tau Z \& \nu P_1 \& \nu P_0.
          \Ar{1-1}{1-2}{}
          \Ar{1-2}{1-3}{"f"}
          \Ar{1-3}{1-4}{"g"}
          \Ar{2-1}{2-2}{}
          \Ar{2-2}{2-3}{hookrightarrow}
          \Ar{2-3}{2-4}{"\nu f_1"'}
          \Ar{1-2}{2-2}{equal}
          \Ar{1-3}{2-3}{"h"}
          \Ar{1-4}{2-4}{"\theta_Z"}
        \end{tikzcd}
      \]}
    \State \Return $\tau Z,\  E_Z$
    \EndFunction
  \end{algorithmic}
\end{algorithm}

\begin{proposition}[Section~3.2 of \cite{gabriel1980auslander}]
  \label{prop:ass-p}
  Let $A$ be a finite-dimensional algebra.
  Then given a non-projective indecomposable $A$-module $Z$ with $\End_A(Z) \cong K$, Algorithm~\ref{alg:ass-p} computes an almost split sequence
  $0 \to X \xrightarrow{f} E_Z \xrightarrow{g} Z \to 0$ in line~\ref{alg:ass-ans}.
\end{proposition}
\begin{proof}
  See Section~3.2 of \cite{gabriel1980auslander}.
\end{proof}

In Algorithm~\ref{alg:ass-p}, we use the following basic concepts.
We recall that $D$ is the $K$-dual given by $D(-) := \Hom_K(-,K): \mmod A \rightarrow \mmod A^{\mathrm{op}}$, and that 
$\nu:= D \circ \Hom_A(\text{-}, A): \mmod A \rightarrow \mmod A$ is the {\em Nakayama functor}. See \cite[Chap.~III.2]{assem2006elements} for more details.

We note that Section 3.2 of \cite{gabriel1980auslander} in fact provides the procedure for any non-projective indecomposable module $Z$. Here, we restrict our attention to $Z$ with $\End{Z} \cong K$ since all interval representations
$Z$ satisfy this condition, and this simplifies the choice of $S$ in Line~\ref{alg:ass-theta} (in general another condition needs to be imposed on $S$).

Below, we go through the steps of Algorithm~\ref{alg:ass-p} given an \emph{interval} representation $Z = V_L$ and analyze its complexity. Furthermore, this restriction to intervals simplifies some of the computation, as we can provide explicit forms for minimal projective presentation
and the map $\nu f_1$.

\paragraph{(Line~\ref{alg:ass-mpp} of Algorithm~\ref{alg:ass-p})} Let $V_L$ be the interval representation associated with the interval subquiver
$
L=\{[b_j,d_j]_j\suchthat s\leq j\leq t\}
$.
Below, we explain the computation of a minimal projective presentation $P_1 \xrightarrow{f_1} P_0 \xrightarrow{\varepsilon} V_L \to 0$.

First, let us review some basic concepts.
Recall that $\Gf{m,n}$ can be regarded as a subposet
of $\bbZ \times\bbZ$ by the order
\[
  (x_1,y_1)\leq (x_2,y_2) \Longleftrightarrow x_1 \leq x_2 \text{ and } y_1 \leq y_2
\]
for any vertices $(x_1,y_1), (x_2,y_2)$ of $\Gf{m,n}$.


For any pair of two vertices 
$a=(x_1,y_1)$ and $b=(x_2,y_2)$ 
of $\Gf{m,n}$, we denote their 
{\em join} (resp. {\em meet}) by 
$a\vee b$ (resp. $a\wedge b$). 
These always exist in $\Gf{m,n}$
and are given by
\[
  a\vee b=(\max(x_1,x_2),\max(y_1,y_2))
  \text{ and }
  a \wedge b =
  (\min(x_1,x_2),\min(y_1,y_2)).
\]

For each interval $L$, 
we fix the representation $V_L$, isomorphic to an interval representation associated to $L$, by the following.
We define
\[
V_L(a) = 
\left\{
\begin{array}{ll}
Ka &\text{if } a \in L_0,\\
0 & \text{otherwise}.
\end{array}
\right.
\]
That is, for each $a \in L_0$, 
we set $V_L(a)$ to be the $K$-vector space of multiples of the vertex $a$ itself (with $a$ as fixed basis).
For each arrow
$\alpha \colon a \to b$ in $L$, the map
$V_L(\alpha) \colon V_L(a) \to V_L(b)$
is defined by $\lambda a \mapsto \lambda b$
for all $\lambda \in K$.

We recall fundamental facts on representations of quivers and modules over algebras by specializing to our case.
Set $A:= K\Gf{m,n}$, and $J$ to be the ideal of $A$ generated by
all arrows in $\Gf{m,n}$. It is well-known that $J$ becomes the Jacobson radical of $A$.
For each $a, b \in (\Gf{m,n})_0$ we denote by $p_{a,b}$ the element of the algebra $A$ represented by a path from $b$ to $a$,
and we set $e_a:= p_{a,a}$.
Note that $p_{a,b}$ is uniquely determined by $a$ and $b$ because $\Gf{m,n}$ has the full commutativity relations.
Then there exists a well-known equivalence between the category
of representations of $\Gf{m,n}$ and the category of
(left) $A$-modules, sending each representation $V$ to
$\bigoplus_{u \in (\Gf{m,n})_0}V(u)$ with the $A$-action defined
by $V(\alpha)$ for all $\alpha \in (\Gf{m,n})_1$, a quasi-inverse of which
sends each $A$-module $M$ to the representation
$(e_a M, \lambda_{\alpha})_{a\in (\Gf{m,n})_0, \alpha \in (\Gf{m,n})_1}$, where $\lambda_\alpha$ denotes the left multiplication by $\alpha$.
By this equivalence
we often identify representations of $\Gf{m,n}$ with their
corresponding (left) $A$-modules.

The sets 
\begin{itemize}
 \item $\{P(a):= Ae_a \mid a \in (\Gf{m,n})_0\}$
 \item $\{S(a):= Ae_a/Je_a \mid a \in (\Gf{m,n})_0\}$
 \item $\{I(a):= \nu P(a) = D(e_aA) \mid a \in (\Gf{m,n})_0\}$
\end{itemize}
respectively form complete sets of representatives of indecomposable projective, simple, and indecomposable injective $A$-modules.
Note that $\{e_a + Je_a\}$, $\{p_{x,a} \mid a \le x\}$,
$\{p_{a,x} \mid a \ge x\}$ form bases of $S(a)$, $P(a)$, and
$e_aA$. Set $\{p_{a,x}^\vee \mid a \ge x\}$ to be the basis of $I(a)$
dual to $\{p_{a,x} \mid a \ge x\}$.
We now give explicit forms for
$P(a), S(a), I(a)$
for all $a\in (\Gf{m,n})_0$ as representations.
%
%
%
%
It is clear that $(\supp S(a))_0= \{a\}$,
$(\supp P(a))_0= \{x \in (\Gf{m,n})_0 \mid a\leq x \}$,
and
$(\supp I(a))_0= \{x \in (\Gf{m,n})_0 \mid a\geq x \}$.
Then as is easily verified, the bijections between bases
\[
e_a + Je_a \mapsto a,\quad
p_{x,a} \mapsto x,\quad p_{a,x}^\vee \mapsto x
\]
define isomorphisms
\begin{equation}
\label{eq:S-P-I-identify}
S(a) \to V_{\{a\}},\quad
P(a) \to  V_{\{x \in (\Gf{m,n})_0 \mid a\leq x\}},\quad
I(a) \to V_{\{x \in (\Gf{m,n})_0 \mid a\geq x\}},
\end{equation}
respectively by which we identify them.
Thus all of $S(a), P(a), I(a) \ (a \in (\Gf{m,n})_0)$ are interval representations, and
we fix bases for their vector spaces by these identifications.

Recall that for each left $A$-module $M$,
right $A$-module $N$ and each $a \in (\Gf{m,n})_0$
we have isomorphisms
\begin{equation}
\label{eq:yoneda}
\begin{aligned}
e_a M &\to \Hom_A(P(a), M),\quad m \mapsto \rho_m,\\
N e_a &\to \Hom_A(e_a A, N),\quad n \mapsto \lambda_n,
\end{aligned}
\end{equation}
where $\rho_m$ is the right multiplication by $m$,
and $\lambda_n$ is the left multiplication by $n$,
that is, $\rho_m(x):= xm\ (x \in P(a))$, and
$\lambda_n(x):= nx\ (x \in e_a A)$.
Note that
$\rho_{p_{a,b}} \colon P(a) = Ae_a \to P(b) = Ae_b$ is sent by
$\Hom_A(-,A)$ to
$\lambda_{p_{a,b}} \colon e_b A \to e_a A$
after identification using the first isomorphism in Eq.~\eqref{eq:yoneda} with $M=A$.
Applying the $K$-dual $D$, 
we obtain 
\begin{equation}
\label{eq:nueps}
\nu(\rho_{p_{a,b}}) = D(\lambda_{p_{a,b}}) \colon I(a)  \to I(b)
\end{equation}
where $\nu(-) := D\Hom(-,A)$ is the Nakayama functor.

We also have
\begin{equation}
\label{eq:yoneda2}
\Hom_A(M, I(a)) \cong D(e_a M)
\end{equation}
for all $a \in (\Gf{m,n})_0$. Indeed,
\[
\begin{array}{rcl}
D(e_a M) &\cong&  D(M)e_a \cong \Hom_A(e_a A, D(M)) \\
          &\cong&  \Hom_A(DD(M), D(e_a A)) \cong \Hom_A(M, I(a)).
          \end{array}
\]
where the second isormorphism follows by the second isomorphism in Eq.~\eqref{eq:yoneda} with $N=D(M)$, and the third isomorphism follows since $D$ is a duality. Alternatively, this follows immediately by tensor-hom adjunction:
\[
\begin{aligned}
D(e_a M) &\cong \Hom_K(e_aA \otimes_A M, K) \\
&\cong \Hom_A(M, \Hom_K(e_a A,K)) \cong \Hom_A(M, I(a)).
\end{aligned}
\]

In particular, 
substituting $M = P(b)$ in Eq.~\eqref{eq:yoneda} and $M = I(b)$ in Eq.~\eqref{eq:yoneda2} for $b \in (\Gf{m,n})_0$,
we have
\[
\Hom_A(P(a), P(b)) = K \rho_{p_{a,b}}, \text{ and }
\Hom_A(I(b), I(a)) = K D(\lambda_{p_{a,b}})= K \nu(\rho_{p_{a,b}}).
\]
Moreover, for $M = V_L$ with $L$ an interval subquiver  of $\Gf{m,n}$ we have
\[
\text{for }a\in L,\;\;
\Hom_A(P(a), V_L) = K \rho_a
\text{ and }
\Hom_A(V_L, I(a)) = K D(\lambda_a)
\]
because $V_L(a) = Ka$ for $a\in L$.
The following is easy to verify.

\begin{lemma}
Let $a \in (\Gf{m,n})_0$ and $L$ be an interval subquiver  of $\Gf{m,n}$.
Then the explicit form for $\rho_a$
and $D(\lambda_a)$ above as morphisms of representations
under the identifications \eqref{eq:S-P-I-identify}
are given as
morphisms $\epsilon_{a,V_L} : P(a) \to V_L$
and 
$\epsilon'_{V_L,a}: V_L \to I(a)$
defined by 
    \[
    \epsilon_{a,V_L} (c) = 
    \left\{ 
    \begin{array}{lcl}
    1_{K c} & & \text{if } a \leq c \text{ and } c\in L_0,\\
    0     & & \text{otherwise},
    \end{array} 
    \right.
    \] 
and by
 \[
    \epsilon'_{V_L,a} (c) = 
    \left\{ 
    \begin{array}{lcl}
    1_{K c} & & \text{if } a \geq c \text{ and } c\in L_0,\\
    0     & & \text{otherwise},
    \end{array} 
    \right.
    \]
respectively.
\end{lemma}

\begin{definition}\label{dfn:mor-prj-inj}
For each $a, b \in (\Gf{m,n})_0$ we set
\[
\left\{
\begin{aligned}
\epsilon_{a,b}&:= \epsilon_{a,P(b)} \colon P(a) \to P(b),\ \text{and}\\
\epsilon'_{a,b}&:= \epsilon'_{I(a),b} \colon I(a) \to I(b).
\end{aligned}
\right.
\]
Thus these give the explicit forms of the morphisms $\rho_{p_{a,b}}$ and $D(\lambda_{p_{b,a}}) = \nu \rho_{p_{a,b}}$ (see Eq.~\eqref{eq:nueps}), respectively.
In particular, we have $\nu(\epsilon_{q,p}) = \epsilon'_{q,p}$.
\end{definition}

For an interval $L=\{[b_j,d_j]_j\suchthat s\leq j\leq t\} \in \II{m,n}$, let $\Sc (L)$ be the set of source vertices in $L$, which is given by
\[
\Sc(L) = \{(b_j, j) \mid j = \min\{l \mid b_l = b_k\}  \text{ for some } k \text{ with } s \le k \le t\}.
\]

To compute a minimal projective presentation of $V_L$, we need the concept of upset, a special type of interval.
The overall strategy is to first compute minimal projective presentations of $V_U$ for upsets $U$ (Proposition~\ref{prp:upsetpropres}). Then, as $V_L$ has the form $V_L \cong V_U/V_U'$ for some upsets $U, U'$ (see Lemma~\ref{lem:interval=upset/upset}), we piece 
together the minimal projective presentations of $V_U, V_U'$ to have that of $V_L$.

\begin{definition}[upsets and upset representations]
A subset $U$ of $(\Gf{m,n})_0$ is called an {\em upset} if the conditions $x \le y$ in $(\Gf{m,n})_0$ and $x \in U$ imply $y \in U$.

Obviously the intersection of any upsets is again an upset.
Therefore for each subset $S$ of $(\Gf{m,n})_0$
there exists the minimum upset $U$ of $(\Gf{m,n})_0$ such that $S \subseteq U$, which we denote by $U(S)$.
When $S = \{a\}$ for some $a \in (\Gf{m,n})_0$,
we simply write $U(a):= U(\{a\})$.

Since for any upset $U$, the full subquiver $\full(U)$ of $\Gf{m,n}$ with
$\full(U)_0 = U$ turns out to be an interval (see the lemma below),
the interval representation $V_U:= V_{\full(U)}$ is defined, and is called an {\em upset representation}.
\end{definition}

The following is obvious.
\begin{lemma}\label{lem:upset}
\leavevmode
\begin{enumerate}
\item 
For all $a \in (\Gf{m,n})_0$,
$U(a) = (\supp P(a))_0$, and $V_{U(a)} = P(a)$.
\item For $S$ a subset of $(\Gf{m,n})_0$,
$
U(S) = \bigcup\limits_{a \in S} U(a) \text{ and }
V_{U(S)} = \sum\limits_{a \in S}P(a).
$
\item For an interval $L$,
\[
U(L_0) = U(\Sc(L)) = \bigcup\limits_{a \in \Sc(L)} U(a)
\text{ and }
V_{U(L_0)} = \sum\limits_{a \in \Sc(L)} P(a)
\]
\item
Every upset is the vertex set of an interval.
\item For an interval $L :=\{[b_j,d_j]_j\suchthat s\leq j\leq t\}$, $L_0$ is an upset
if and only if
$t = n$ and $d_i=m$ for all $s\leq i \leq n$.
\end{enumerate}
\end{lemma}

We now give a minimal projective presentation of an upset representation.

\begin{proposition}\label{prp:upsetpropres}
Let $U$ be an upset of $(\Gf{m,n})_0$, and
set $\{p_1,\dots p_l\}:= \Sc(U)$, and
$q_d:= p_d \vee p_{d+1}$ for all $d = 1,\dots, l-1$,
where $p_c = (b_{j_c}, j_c)$, $(c = 1,\dots, l)$, with $n \ge j_1 > j_2 > \dots > j_l \ge 1$.
Then a minimal projective presentation of $V_U$
is given by
\begin{equation}\label{eq:min-proj-res-upset}
0\to \bigoplus_{d=1}^{l-1} P(q_d) \xrightarrow{f_U}
\bigoplus_{c=1}^l P(p_c) \xrightarrow{\pi_U} V_U \to 0,
\end{equation}
where $\pi_U = (\epsilon_{p_c, V_U})_{c=1}^l$ and
\[f_U =
\begin{pmatrix}
\epsilon_{q_1,p_1} &&&\\
-\epsilon_{q_1, p_2} & \epsilon_{q_2, p_2}&&\\
 & -\epsilon_{q_2, p_3} & \quad\ddots&\\
 & & \quad\ddots & \quad \epsilon_{q_{l-1},p_{l-1}}\\
 &&&\quad -\epsilon_{q_{l-1}, p_l} 
\end{pmatrix}
\quad
 \text{(each blank entry is zero)}.
\]
\end{proposition}

\begin{proof}
We first show that the equality
\begin{equation}\label{eq:dim-additive}
\sum_{c=1}^l\dim P(p_c) - \sum_{d=1}^{l-1}\dim P(q_d) = \dim V_U
\end{equation}
holds.
It is enough to show the equality
\begin{equation}\label{eq:dimvec-additive}
\sum_{c=1}^l\dim P(p_c)(p) - \sum_{d=1}^{l-1}\dim P(q_d)(p) = \dim V_U(p)
\end{equation}
for all vertices $p \in (\Gf{m,n})_0$.
If $p \not\in U$, then $\dim V_U(p) = 0$ by definition, and we have
$p \not\ge p_c$
for all $c =1,\dots,l$, and hence
$p \not \ge q_d$
for all $d=1,\dots, l-1$.
Thus the left hand side is also zero, and Eq.~\eqref{eq:dimvec-additive} holds.
Assume $p \in U$.
We set $\{p_{j_1},\dots, p_{j_t}\}:= \{p_c \mid p_c \le p\}$ with ${j_1} > \dots > {j_t}$. Note that the indices $j_i$ are contiguous integers.
Then $\sum_{c=1}^l\dim P(p_c)(p) = \#\{p_c \mid p_c \le p\} = t$, and
$\sum_{d=1}^{l-1}\dim P(q_d)(p) = \#\{q_d \mid q_d \le p\} = \#\{q_{j_1},\dots, q_{j_{t-1}}\} = t - 1$.
Since
$\dim V_U(p) = 1 = t - (t-1)$, Eq.~\eqref{eq:dimvec-additive} holds also in this case, and hence
the equality \eqref{eq:dim-additive} is verified.

Now since all $\epsilon_{q_c,p_c}$ are monomorphisms,
$f_U$ is also a monomorphism.
On the other hand, $\pi_U$ is an epimorphism because
$\Image \pi_U = \sum_{c=1}^l \Image \epsilon_{p_c, V_U} = \sum_{c=1}^l P(p_c) = V_U$
by Lemma \ref{lem:upset}(3).
Furthermore, since
$\epsilon_{p_c,V_U}\epsilon_{q_c,p_c} = \epsilon_{p_{c+1},V_U}\epsilon_{q_c,p_{c+1}}$,
we have $\pi_U f_U = 0$.
These facts, together with the equality \eqref{eq:dim-additive}
show that the sequence \eqref{eq:min-proj-res-upset} above is exact.

Obviously $\pi_U$ induces an isomorphism between the tops, and hence
it is a projective cover of $V_U$.
The exactness of the sequence \eqref{eq:min-proj-res-upset} shows that
$f_U$ is a projective cover of $\Ker \pi_U$.
Therefore the sequence \eqref{eq:min-proj-res-upset} is a minimal projective presentation of $V_U$.

\end{proof}

\begin{lemma}\label{lem:interval=upset/upset}
Let $L=\{[b_j,d_j]_j\suchthat s\leq j\leq t\}\in \II{m,n}$ be an interval.
Define 
\[
U:= U(L) = \{[b_j,m]_j \mid s \le j \le n\} 
\]
where $b_j:= b_t$ for $j = t+1, \dots n$, and
\[
U':= \{[d_j + 1, m]_j \mid s \le j \le n\}.
\]
Then $U$ and $U'$ are upsets satisfying
\[
V_L \cong V_U/V_{U'}.
\]
\end{lemma}

\begin{proof}
Both $U$ and $U'$ are upsets by Lemma \ref{lem:upset}(5).
The statement follows from the following calculations:
\[
\begin{aligned}
L_0 &= \bigcup_{j=s}^t[b_j,d_j]_j
=\bigcup_{j=s}^t ([b_j,m]_j \setminus [d_j + 1, m]_j)\\
&= \left(\bigcup_{j=s}^t [b_j,m]_j\right) \setminus \left(\bigcup_{j=s}^t [d_j + 1, m]_j\right)
= \left(\bigcup_{j=s}^n [b_j,m]_j\right) \setminus \left(\bigcup_{j=s}^n [d_j + 1, m]_j\right)
= U \setminus U'.
\end{aligned}
\]
\end{proof}

\begin{proposition}\label{prp:min-prj-pres-interval}
Let $L$ be an interval of $\Gf{m,n}$ and
$U$, $U'$ the upsets defined in Lemma \ref{lem:interval=upset/upset}.
Set
$\{p_1,\dots, p_l\}:= \Sc(L) = \Sc(U)$
and $q_d:= p_d \vee p_{d+1}$ for all $d=1, \dots, l-1$, where $p_c = (b_{j_c},j_c)$, $(c =1,\dots,l)$, with $n \ge j_1 > \dots > j_l \ge 1$. Set also $c(r):= \min\{c \mid p_c \le r\}$
for all $r \in \Sc(U')$.
Then a minimal projective presentation of $V_L$ is given by
\begin{equation}\label{eq:min-prj-pres-interval}
\bigoplus\limits_{r \in \Sc(U')} P(r) \oplus \bigoplus\limits_{d=1}^{l-1} P(q_d)\xrightarrow{(f'_L, f_{U})}
\bigoplus_{c=1}^l P(p_c) \xrightarrow{\pi_L}
V_L \to 0,
\end{equation}
where $\pi_L:= (\epsilon_{p_c,V_L})_{c=1}^l$,
and $f'_L:= (\delta_{c,c(r)}\epsilon_{r,p_{c(r)}})_{c,r}$.
Here, $\delta_{ij}$ is the Kronecker delta.
\end{proposition}

\begin{proof}
For simplicity we put $P_0:= \bigoplus_{c=1}^l P(p_c)$ and
$P_1:= \bigoplus_{d=1}^{l-1} P(q_d)$.
Then we have an exact sequence
\[
0 \to P_1 \xrightarrow{f_U} P_0 \xrightarrow{\pi_U} V_U \to 0,
\]
which is a minimal projective presentation of $V_U$ by
Proposition \ref{prp:upsetpropres}.
By the same way we construct a minimal projective presentation of $V_{U'}$ of the form
\[
0 \to P'_1 \xrightarrow{f_{U'}} P'_0 \xrightarrow{\pi_{U'}} V_{U'} \to 0,
\]
where we note that $P'_0:= \bigoplus_{r \in \Sc(U')} P(r)$.
Let $v \colon V_U \to V_L$ be the epimorphism defined by
\[
  v(x) =
    \begin{cases}
      1_{K x}, & x \in L_0 \\
      0,     & \text{otherwise}.
    \end{cases}
\]
Then this induces an isomorphism between the tops, and hence
$\pi_L:= v\pi_U \colon P_0 \to V_L$ is a projective cover of $V_L$.
Set $\Omega V_L:= \Ker \pi_L$ and let $\mu \colon \Omega V_L \to P_0$ and
$u \colon V_{U'} \to V_U$ be the inclusions.
Then there exist unique morphisms $g, g'$ that make the following diagram commute, with exact rows and columns:
\begin{equation}\label{eq:Om-VL}
\vcenter{
  \xymatrix{
  & 0 \ar[d] & 0 \ar[d]\\
    & P_1 \ar@{=}[r] \ar[d]_{g} & P_1 \ar[d]^{f_U} \\
    0 \ar[r] & \Omega V_L \ar@{^{(}->}[r]^{\mu} \ar[d]_{g'} & P_0 \ar[r]^{\pi_L} \ar[d]^{\pi_U} & V_L \ar[r] \ar@{=}[d] &0 \\
    0 \ar[r] & V_{U'} \ar@{^{(}->}[r]_{u} \ar[d]& V_U \ar[r]_{v} \ar[d]& V_L \ar[r] &0\\
    & 0 & 0
  }}.
\end{equation}
Consider the following diagram of solid arrows with exact rows:
\begin{equation}\label{eq:po-pb}
\vcenter{
\xymatrix{
0 & P'_1 & P'_0 & V_{U'} & 0\\
0 & P_1 & \Omega V_L & V_{U'} & 0
\ar"1,1";"1,2"
\ar"1,2";"1,3"^{f_{U'}}
\ar"1,3";"1,4"^{\pi_{U'}}
\ar"1,4";"1,5"
\ar"2,1";"2,2"
\ar"2,2";"2,3"_g
\ar"2,3";"2,4"_{g'}
\ar"2,4";"2,5"
\ar@{-->}"1,2";"2,2"_{h'}
\ar@{-->}"1,3";"2,3"^h
\ar@{=}"1,4";"2,4"
}}
\end{equation}
By the projectivity of $P'_0$ this is completed to
the commutative diagram with $h, h'$.
We may take $h$ in such a way that
$\mu h \colon P'_0 \to P_0$ is given by the matrix
$f'_L:= (\delta_{c,c(r)}  \epsilon_{r,p_{c(r)}})_{c,r}$.
Indeed, since $u$ is a monomorphism, the equality $ug'h = \pi_U \mu h
= \pi_U (\delta_{c,c(r)}  \epsilon_{r,p_{c(r)}})_{c,r}
\overset{*}{=} u\pi_{U'}$ shows that $g'h = \pi_{U'}$, where the last equality $(\overset{*}{=})$ holds because
the restrictions of both sides to $P(r)$ coincide for all $r \in \Sc(U')$.

Since the left square of the diagram \eqref{eq:po-pb} is a pushout and pullback diagram,
we have the following exact sequence:
\[
0 \to P'_1 \xrightarrow{\left(\begin{smallmatrix}
f_{U'}\\ -h'
\end{smallmatrix}\right)} P'_0 \oplus P_1 \xrightarrow{(h,g)} \Omega V_L \to 0.
\]
Here $(h,g)$ is a projective cover of $\Omega V_L$.
Indeed, since $\pi_{U'}$ is a projective cover of $V_{U'}$,
we have $\Image f_{U'} \subseteq \rad P'_0$, and
by the form of $\mu h = f'_L$ we have $\Image h'\subseteq \rad P_1$.
Therefore $\Ker (h,g) = \Image \begin{pmatrix}
f_{U'}\\-h'
\end{pmatrix} \subseteq \rad(P'_0 \oplus P_1)$, as required.
By connecting this sequence and the upper horizontal short exact sequence
in the diagram \eqref{eq:Om-VL} we obtain a minimal projective presentation
\[
P'_0 \oplus P_1 \xrightarrow{\mu (h,g)} P_0 \xrightarrow{\pi_L} V_L \to 0
\]
of $V_L$.
Here note that $\mu (h,g) = (f'_L, f_U)$.

\end{proof}

\paragraph{\bf Complexity Analysis for Line~\ref{alg:ass-mpp} of Algorithm~\ref{alg:ass-p}.}
We let $l = |\Sc(L)| = |\Sc(U)|$
and
$l':= |\Sc(U')|$. Furthermore, we set $z:= \min\{m,n\}$ and assume that $n = z = \min\{m,n\}$. Note that $l, l' \le z$.

We give the cost of calculating (symbolically) the minimal projective presentation of $V_L$ as given by Proposition~\ref{prp:min-prj-pres-interval}. 
For this, we need
to compute $U$, $U'$ (Lemma~\ref{lem:interval=upset/upset}) and their source vertices $\Sc(U')$ and 
$\Sc(U) = \{p_1,\dots, p_l\}$ where $p_c = (b_{j_c},j_c)$ for $c =1,\dots,l$, with $n \ge j_1 > \dots > j_l \ge 1$.
Then we need to compute $q_d:= p_d \vee p_{d+1}$ for all $d=1, \dots, l-1$, and $c(r) = \min\{c \suchthat p_c \leq r\}$ for each $r \in \Sc(U')$.

\begin{itemize}
\item First, the computation of $U$ and $U'$ from $L$ follows using Lemma~\ref{lem:interval=upset/upset}. This costs $O(z)$ by an obvious iteration over rows.
\item Next, let us give the cost of calculating $\Sc(U)$ for an upset $U$.
Let $U:=\{[b_j,m]_j \suchthat s \leq j \leq n\}$ be an upset.
Then we iterate over the rows starting from the bottom row and going up. First, we record $(b_s,s)$ as a source. Then we iterate $j = s+1, s+2, \hdots, n$, and whenever $b_{j} < b_{j-1}$, we record $(b_{j},j)$ as a source. This costs $O(z)$.
\item For each $d=1,\hdots, l-1$, the calculation of $q_d = p_d \vee p_{d+1}$ costs $O(1)$. This adds up to $O(l)$.
\item For each $r \in \Sc(U')$, the  computation of $c(r) = \min\{c \suchthat p_c \leq r\}$ can be performed via binary search, costing $O(\log(l))$. Thus, overall the computation of $c(r)$ for all $r \in \Sc(U')$ costs $O(l' \log(l))$.
\end{itemize}

Thus, overall we have a cost of $O(z + l + l'\log(l)) \leq O(z\log(z))$.

This ends our discussion and analysis of Line~\ref{alg:ass-mpp} of Algorithm~\ref{alg:ass-p} for the computation of a minimal projective presentation ending at an interval representation $V_L$. Let us move on to the next line.

\paragraph{(Line~\ref{alg:ass-nak} of Algorithm~\ref{alg:ass-p})} 
Next, we compute $\nu f_1 : \nu P_1 \rightarrow \nu P_0$.



By Proposition~\ref{prp:min-prj-pres-interval}
the morphism $f_1$ in the minimal projective presentation of $V_L$ has the form
\[
\bigoplus_{r \in \Sc(U')} P(r) \oplus \bigoplus_{d=1}^{l-1} P(q_d)\xrightarrow{f_1 = (f'_L, f_{U})}
\bigoplus_{c=1}^l P(p_c),
\]
where
\[
f'_L:= (\delta_{c,c(r)}\epsilon_{r,p_{c(r)}})_{c,r},\quad \text{and}\quad
f_U =
\begin{pmatrix}
\epsilon_{q_1,p_1} &&&\\
-\epsilon_{q_1, p_2} & \epsilon_{q_2, p_2}&&\\
 & -\epsilon_{q_2, p_3} & \quad\ddots&\\
 & & \quad\ddots & \quad \epsilon_{q_{l-1},p_{l-1}}\\
 &&&\quad -\epsilon_{q_{l-1}, p_l} 
\end{pmatrix}.
\]
By $\nu$ this is sent to
\[
\bigoplus_{r \in \Sc(U')} I(r) \oplus \bigoplus_{d=1}^{l-1} I(q_d)\xrightarrow{(\nu f'_L, \nu f_{U})}
\bigoplus_{c=1}^l I(p_c),
\]
where
\begin{equation}
\nu f'_L:= (\delta_{c,c(r)}\epsilon'_{r,p_{c(r)}})_{c,r},\quad \text{and}\quad
\nu f_U =
\begin{pmatrix}
\epsilon'_{q_1,p_1} &&&\\
-\epsilon'_{q_1, p_2} & \epsilon'_{q_2, p_2}&&\\
 & -\epsilon'_{q_2, p_3} & \quad\ddots&\\
 & & \quad\ddots & \quad \epsilon'_{q_{l-1},p_{l-1}}\\
 &&&\quad -\epsilon'_{q_{l-1}, p_l} 
\end{pmatrix}
\label{eq:nuf1}
\end{equation}
by using 
the remark in Definition \ref{dfn:mor-prj-inj}.
Note that there is no need for new computations here. In the previous step we have symbolically calculated the minimal projective presentation
by calculating $\Sc(U) = \{p_1,\dots, p_l\}$, $\Sc(U')$, and $q_d = p_d \vee p_{d+1}$ for $d = 1,\hdots, l-1$, and $c(r) = \min\{c \suchthat p_c \leq r\}$ for each $r \in \Sc(U')$,
at total cost of $O(z\log(z))$. In this step we essentially only turned $\epsilon_{q,p}$ to $\epsilon'_{q,p}$.

\paragraph{(Line~\ref{alg:ass-trans} of Algorithm~\ref{alg:ass-p})}
  In this part, we need to compute
  \[
    \tau L:= \Ker(\nu f_1 : \nu P_1 \rightarrow \nu P_0)
  \]
  
  First, we need to express $\nu f_1 = (\nu f'_L, \nu f_U) : \nu P_1 \rightarrow \nu P_0$, 
  so far computed only
  symbolically, in terms of vector spaces and linear maps (as a representation).
  
  The entries of $\nu f_1 = (\nu f'_L, \nu f_U)$ involve the morphisms of the form  $\epsilon'_{q,p}$ (see Equation~\eqref{eq:nuf1}). 
  Fix one such $\epsilon'_{q,p}$.
  For each vertex $v \in (\Gf{m,n})_0$, we compare $v$ with $p$ and $q$. If $v \leq p$ and $v \leq q$, then we put the scalar $1_K$ in the appropriate entry. Over all vertices, this operation costs $O(mn)$ in total. 
  Then, since $\nu f_U$ contains $2(l-1)$ entries and $\nu f'_L$ contains $l' = |\Sc(U')|$ entries involving $\epsilon'_{q,p}$, expressing $\nu f_1$ as a collection of matrices costs $O(mn((l+l')) \leq O(mnz)$.
 
 For the computation of the kernel, we also need the internal maps $(\nu P_1)(\alpha)$ for all $\alpha \in (\Gf{m,n})_1$. We let $S_1:= \Sc(U') \cup \{q_d \mid 1 \le d \le q-1\}$.
 Then
 \[
 \nu P_1 = 
 \bigoplus_{r \in \Sc(U')} I(r) \oplus \bigoplus_{d=1}^{l-1} I(q_d) 
 = \bigoplus_{r \in S_1} I(r).
 \]
 For fixed arrow $\alpha$ in $\Gf{m,n}$, 
 let $a = \#\{r \in S_1 \mid s(\alpha) \le r\}$
and
$b= \#\{r \in S_1  \mid t(\alpha) \le r\}$. 
We have
 $(\nu P_1)(\alpha) = \bigoplus\limits_{r \in S_1} I(r)(\alpha) \colon K^a \to K^b$,
 and for each $r \in S_1$,
\begin{equation}
    \label{eq:injective_internal}
    I(r)(\alpha) :
     \begin{cases}
     (Ks(\alpha) \to Kt(\alpha)): 
     [\lambda s(\alpha) \mapsto \lambda t(\alpha)], \lambda \in K,
     & 
     s(\alpha), t(\alpha) \le r,
     \\
     0 \to Kt(\alpha), & s(\alpha) \not \le r, t(\alpha) \le r\\
     Ks(\alpha) \to 0, & s(\alpha)\le r, t(\alpha) \not \le r\\
     0 \to 0, & \text{otherwise}.
     \end{cases}
\end{equation}

Then for each $r \in S_1$, we determine the row and column in the $b\times a$ matrix corresponding to $r$. Note that only in the case of $s(\alpha) \leq r$ and $t(\alpha) \leq r$ will there be a corresponding entry. In that case, we put a $1$ in the matrix. The rest of the entries of the matrix are $0$. 
Since $\# S_1 = l' + l-1$, this costs $O(l' + l)$ for each $\alpha$. Then, since there are $O(mn)$ arrows, we get a total cost of $O(mn(l'+l)) \leq O(mnz)$.

Having expressed $\nu f_1: \nu P_1 \rightarrow \nu P_0$ in terms of vector spaces and linear maps,
next we discuss the computation of $\Ker \nu f_1$.
In general, for a linear map 
$\phi \colon K^p \rightarrow K^q$,
we can get an injection $\sigma_{\phi} \colon \Ker\phi \rightarrow K^p$ by performing column operations on the augmented matrix:
\[
\left(
\begin{array}{c}
    \rule{3pt}{0pt}
    \phi
    \rule{3pt}{0pt}
    \\ 
    \hline
    I_p
\end{array}
\right)
\overset{\text{col ops}}
{\rule{0pt}{8pt}\rightsquigarrow}
\left(
\begin{array}{c|c}
    \smash{\overbrace{
    \text{col echelon form}}^{\smash{\rk\phi \text{ columns}}}    }
     & \rule{3pt}{0pt}
     0
     \\
    \hline 
    & 
    \rule{3pt}{0pt}
    \sigma_{\phi}
\end{array}
\right)
\]
where $I_p$ denotes the identity matrix of size $p$.
Since $\sigma_{\phi}$ is a section map, 
there exists the retraction $\sigma'_{\phi}$ such that $\sigma'_{\phi}\sigma_{\phi}=I_{\rk\sigma_{\phi}}$.
This $\sigma'_{\phi}$ is also obtained by the following elementary transformations of the matrix:
\[
\left(
    \begin{array}{c|c}
    \rule{2pt}{0pt}
    \sigma_{\phi}
    \rule{3pt}{0pt} 
    & 
    \rule{3pt}{0pt}
    I_p
    \end{array}
\right)
\overset{\text{row ops}}
{\rule{0pt}{8pt}\rightsquigarrow}
\left(
\begin{array}{c|c}
I_{\rk\sigma_{\phi}}
 & 
 \rule{3pt}{0pt}
 \sigma'_{\phi}
 \\
\hline %
\rule{0pt}{15pt}
0
&
\end{array}
\right).
\]

Hence, for a morphism $F\colon M \rightarrow N$ in $\rep (Q,R)$, we can compute $\Ker F$.
For each vertex $v\in Q_0$, we have
$(\Ker F)(v):=\Ker(F_v)=K^{\rk\sigma_{F_v}}$ and for each arrow $\alpha\colon u\rightarrow v$ in $Q$, we have
\[
(\Ker F)(\alpha):=\sigma'_{F_v}M(\alpha)\sigma_{F_u}.
\] 
Namely, $\Ker F$ is constructed to make the following diagram commutative.
\[ \xymatrix{
 0 \ar[r] & (\Ker F)(u) \ar[r]^{\sigma_{F_u}} \ar@{.>}[d]_{(\Ker F)(\alpha)}& M(u) \ar[r]^{F_u} \ar[d]^{M(\alpha)} & N(u) \ar[d]^{N(\alpha)} \\
 0 \ar[r] & (\Ker F)(v) \ar[r]_{\sigma_{F_v}} & M(v) \ar[r]_{F_v} & N(v) \mathrlap{.}
} \]

Note that in our setting of computing $\Ker \nu f_1$, 
we have $q = \dim \nu P_0(u) \le z$ and
$p = \dim \nu P_1(u) \le 2z$ for all $u \in (\Gf{m,n})_0$.
Then the computation of 
$
(\Ker \nu f_1)(v):=\Ker((\nu f_1)_v)
$ 
for all vertices $v$ costs
$O(z^{\omega}mn)$ 
via column echelon form computations. 
Furthermore, the computation of the internal maps 
$
(\Ker \nu f_1)(\alpha):=
\sigma'_{(\nu f_1)_v}
[(\nu P_1)(\alpha)]
\sigma_{(\nu f_1)_u}
$ 
for all arrows $\alpha$ costs $O(z^\omega mn)$ total via matrix multiplications.

\paragraph{(Line~\ref{alg:ass-theta} of Algorithm~\ref{alg:ass-p})}
For $Z = V_L$ recall that $\theta_L:= \theta_Z$ is given by the composite
  \[
    \theta_Z \colon Z \xrightarrow{\text{\rm can}} \ttop Z  \xrightarrow{\pi} S \hookrightarrow \ttop Z   \xrightarrow{\overline{\varepsilon}^{-1}} \ttop P_0 \xrightarrow{\sim} \soc \nu P_0,
  \]
  where $S$ is any simple direct summand of $\ttop{Z}$,
  and
  $\pi \colon \ttop{Z} \to S$,  $\ttop P_0 \xrightarrow{\sim} \soc \nu P_0$ are the canonical isomorphisms.
  Furthermore, the isomorphism
  $\ttop Z \xrightarrow{\overline{\varepsilon}^{-1}} \ttop P_0$ is the one induced by $\varepsilon : P_0 \rightarrow Z$.
Now let $L=\{[b_j,d_j]_j\suchthat s\leq j\leq t\}\in \II{m,n}$, and set
$\Sc(L) = \{p_1,\dots, p_l\}$ as in Proposition \ref{prp:min-prj-pres-interval}.
Then we may take $S:= S(p_1)$, and $\theta_L$ can be defined by
\[
\theta_L:= {}^t\!\!(\delta_{c,1}\epsilon'_{V_L,p_1})_{c=1}^l \colon V_L \to \nu P_0= \bigoplus_{c=1}^l I(p_c).
\]

For each vertex $u \in L_0$, we have
\[
(\theta_L)_u = \binom{(\epsilon'_{V_L,p_1})_u}{0}: V_L(u) \rightarrow I(p_1)(u) \oplus \bigoplus_{c=2}^l I(p_c)(u)
\]
where the entry $\epsilon'_{V_L,p_1}(u)$ is $1$ if $p_1 \ge u$ and $0$ otherwise. Computing over all vertices, we have a total cost of $O(mn)$.

\paragraph{(Line~\ref{alg:ass-ans} of Algorithm~\ref{alg:ass-p})}
Note that the middle term $E_L:= E_Z$, a pullback, can be computed as the kernel of
\[
\nu P_1 \oplus V_L \xrightarrow{(-\nu f_1, \theta_L)} \nu P_0.
\]
We can compute this kernel by the method explained above, or instead, we can build it using information we already have.


Obviously we have $E_L \supseteq \Ker \nu f_1 \oplus \Ker \theta_L = \tau V_L \oplus \Ker \theta_L$.
Let $S = S(p_1)$ be the simple direct summand of $\ttop{Z}$ chosen above (in Line~\ref{alg:ass-theta} of Algorithm~\ref{alg:ass-p}), and let $w:= \binom{p_1}{p_1} \in \nu P_1 \oplus V_L$,
where the first entry $p_1$ is $p_1 \in I(q_1)$ in 
\[
 \nu P_1 = 
 \bigoplus_{r \in \Sc(U')} I(r) \oplus \bigoplus_{d=1}^{l-1} I(q_d),
\]
and the second entry $p_1$ is the obvious $p_1 \in V_L$.
Then $w$
is in $E_L$ because
\[
(\nu f_1)_{p_1}(p_1) = ((\nu f'_L)_{p_1}, (\nu f_U)_{p_1})(p_1)
= (\epsilon'_{q_1,p_1})_{p_1}(p_1) = p_1 = (\theta_L)_{p_1}(p_1).
\]
From the exact sequences of the forms
\[
0 \to \tau V_L \to E_L \to V_L \to 0
\quad\text{and}\quad
0 \to \Ker \theta_L \to V_L \to S \to 0
\]
we have 
$\dim E_L = \dim \tau V_L + \dim \Ker \theta_L + 1$.
Therefore noting that $w \not\in \tau V_L \oplus \Ker \theta_L$, we have
$E_L = \tau V_L \oplus \Ker \theta_L \oplus Kw$ as a vector space.
Let $L'$ be the interval $\full(L_0\setminus \{p_1\})$.
Then we have $\Ker \theta_L = V_{L'}$, and hence we finally have
$E_L = \tau V_L \oplus V_{L'} \oplus Kw$.

The representation structure of  $E_L$ is defined by those of $\tau V_L$,
$V_{L'}$ and that of $Kw$ defined by
$E_L(\alpha_{p_1})w:= \binom{0}{V_{L}(\alpha_{p_1})(p_1)}$, $E_L(\beta_{p_1})w:= \binom{\tau V_L(\beta_{p_1})(p_1)}{V_{L}(\beta_{p_1})(p_1)}$,
where $\alpha_{p_1}, \beta_{p_1}$ are the horizontal arrow and the vertical arrow
of $\Gf{m,n}$ starting from $p_1$, respectively.
Namely, for each arrow $\alpha \colon s \to t$ in $\Gf{m,n}$,
$E_L(\alpha)$
is given by
\[
E_L(\alpha) = 
\begin{pmatrix}
\tau V_L(\alpha) & 0 \\
0 & V_{L'}(\alpha) 
\end{pmatrix}
\colon \tau V_L(s) \oplus V_{L'}(s)
\to \tau V_L(t) \oplus V_{L'}(t) 
\]
for arrows $\alpha: s \rightarrow t$ with $s \neq p_1$ and $t \neq p_1$; 
\[
E_L(\alpha) = 
\begin{pmatrix}
\tau V_L(\alpha) & 0 \\
0 & V_{L'}(\alpha) \\
0 & 0
\end{pmatrix}
\colon \tau V_L(s) \oplus V_{L'}(s) 
\to \tau V_L(t) \oplus V_{L'}(t) \oplus (Kv)(t)
\]
for arrows $\alpha$ ending at $t=p_1$, and finally for arrows starting at $s=p_1$ (given by the arrows $\alpha_{p_1}$ and $\beta_{p_1}$), 
\[
E_L(\alpha_{p_1}) = 
\begin{pmatrix}
\tau V_L(\alpha_{p_1}) & 0 & 0\\
0 & V_{L'}(\alpha_{p_1}) & V_{L}(\alpha_{p_1}) \pi'_2
\end{pmatrix}
\colon \tau V_L(s) \oplus V_{L'}(s) \oplus (Kv)(s)
\to \tau V_L(t) \oplus V_{L'}(t) 
\]
and 
\[
E_L(\beta_{p_1}) = 
\begin{pmatrix}
\tau V_L(\beta_{p_1}) & 0 & \tau V_L(\beta_{p_1}) \pi'_1\\
0 & V_{L'}(\beta_{p_1}) & V_{L}(\beta_{p_1}) \pi'_2
\end{pmatrix}
\colon \tau V_L(s) \oplus V_{L'}(s) \oplus (Kv)(s)
\to \tau V_L(t) \oplus V_{L'}(t),
\]
where we set
$\pi_1 \colon \nu P_1 \oplus V_L \to \nu P_1$, and
$\pi_2 \colon \nu P_1 \oplus V_L \to V_L$ to be the canonical projections,
and since $\pi_1(E_L) \le \pi_1(\tau V_L) + \pi_1(\Ker \theta_L) + \pi_1(Kv)
\le \tau V_L$, they restrict to the morphisms
$\pi'_1:= \pi_1|_{E_L} \colon E_L \to \tau V_L$ and $\pi'_2:= \pi_2|_{E_L} \colon E_L \to V_L$.

Note that we have already computed the maps $\tau V_L(\alpha)$, and we only need to copy the known information to create $E_L$.
Thus, for the computational complexity,
we only estimate the size of $E_L$, which is given by
$\sum_{\alpha \in (\Gf{m,n})_1}\dim E_L(s(\alpha))\dim E_L(t(\alpha)) \le mn l^2 \le mnz^2$.

The above arguments show that
\begin{proposition}
\label{prop:costass}
  For $Z=V_L$ a non-projective interval representation,
  Algorithm~\ref{alg:ass-p}, which computes the terms of the almost split sequence ending at $Z$, can be performed in time complexity $O(mnz^\omega)$.
\end{proposition}

\begin{proposition}
  For $L$ an interval,
  Algorithm~\ref{alg:multiplicity}, which computes the multiplicity of $L$ in $M$, can be performed in time complexity $O\left(((\dim M)^\omega+mn)z^\omega\right)$.
\end{proposition}

\begin{proof}
  In case that $L$ is projective, $L$ has support consisting of all vertices that have a directed path from $g$, for some fixed vertex $g$. The module $\rad{L}$ is simply the interval whose support is the support of $L$ with $g$ excluded. 
  We then set $E_L = \rad{L}$ and $\tau L = 0$.
 Otherwise, if $L$ is not projective, we use Algorithm~\ref{alg:ass-p} to compute $\tau L$ and $E_L$.

  Next, we need to compute:
  \[
    d_M(L) = \dimhom{M}{\tau L} - \dimhom{M}{E_L} + \dimhom{M}{L}.
  \]
  Thus, we need to compute $\dimhom{M}{Y}$ for $Y$ equal to $\tau L$, $E_L$, and $L$.
  Remark~2 of \cite{Asashiba2017} shows that
  for $M$ and $Y$ representations of a bound quiver $(Q,R)$,
  $\Hom(M,Y)$ is isomorphic (as $K$-vector space)
  to the kernel of some $D_1 \times D_0$ matrix $B$,
  where
  \[
    D_1 = \sum_{\alpha: i \rightarrow j \text{ in } Q_1} \dim M(i)\dim Y(j)
  \]
  and
  \[
    D_0 = \sum_{i \in Q_0} \dim M(i)\dim Y(i).
  \]
  In particular, we compute $\dim\Hom(M,Y) = D_0 - \rk B$.

  Let us analyze the size of $B$, which depends on $Y$. Let $\Upsilon = \max_{i \in Q_0} \dim Y(i)$. Then
  \[
    \begin{array}{rcl}
      D_1 & \leq &
                   \sum\limits_{\alpha:i\rightarrow j}
                   \dim M(i) \Upsilon
                   \leq 2 \sum\limits_{i\in Q_0}\dim M(i) \Upsilon
                   = 2 \Upsilon \dim M
      \\
      D_0 & \leq&
                  \sum\limits_{i \in Q_0}
                  \dim M(i) \Upsilon
                  = \Upsilon \dim M
    \end{array}
  \]
  where the factor $2$ for $D_1$ comes from the fact that in $\Gf{m,n}$, for each vertex $i$ there are at most $2$ arrows starting from $i$. We then note that using Gaussian elimination, $\rk B$ can be computed in time $O(2\Upsilon^\omega (\dim M)^\omega)$ \cite{ibarra1982generalization}.

\begin{itemize}
  \item In the case that $Y = \rad L$ or $Y = L$, $\Upsilon =1$. 
  
  \item In the case that $Y = \tau L$, we give an upper bound for $\Upsilon$ as below.
  We note that
  \[
    \Upsilon = \max_{i \in Q_0} \dim \tau L(i) \leq
     \max_{i \in Q_0} \dim \nu P_1(i) = 
     \dim \nu P_1(1,1)= \dim P_1(m,n)
 \]
 since $\tau L = \ker (\nu f_1:\nu P_1\rightarrow \nu P_0)$. Furthermore, the maximum of $\dim \nu P_1$ occurs at the bottom-left corner $(1,1)$ since $\nu P_1$ is injective. The final equality follows from the definition of $\nu$. 
 Then since $P_1 = \bigoplus_{r \in \Sc(U')} P(r) \oplus \bigoplus_{d=1}^{l-1} P(q_d)$ by Proposition~\ref{prp:min-prj-pres-interval}, we see that 
\[
    \dim P_1(m,n) = l' + l-1
\]
where $l' =\# \Sc (U')$ and $l =\# \Sc (U)$. Thus $\Upsilon \leq l' + l - 1$ for 
  $Y = \tau L$.
  
  \item Finally, for the case $Y= E_L$, the middle term of the almost split sequence, we have $\Upsilon \leq l' + l$. To see this, note that
  $
    \dim E_L(i) = \dim \tau L(i) + \dim L(i) 
  $
  so that 
  \[
    \Upsilon = \max_{i \in Q_0} \dim E_L(i) \leq \max_{i\in Q_0}\dim \tau L(i) + \max_{i\in Q_0} \dim L(i) \leq (l' +l -1) + 1
  \]
  using the previous case.
\end{itemize}
  
  Recall that $l'\leq z$ and $l\leq z$, where $z = \min\{m,n\}$.
  Combining the above, we have a time complexity of
  \[\begin{array}{rcl}
    && O\left(2(1+ (l' + l - 1)^\omega + (l' + l)^\omega)(\dim M)^\omega\right) + O(A)\\ 
    &=& O\left(z^\omega(\dim M)^\omega\right) + O(mnz^\omega)
    \end{array}
  \]
  if $L$ is nonprojective,
  where $O(A) = O(mnz^\omega)$ is the cost of computing the terms of the almost split sequence as given in Proposition~\ref{prop:costass},
  and $O(z^\omega(\dim M)^\omega)$
  if $L$ is projective.
\end{proof}

In the worst case, we need to test all $\card{\II{m,n}}$ intervals, and we obtain the following:
\begin{theorem}
  Algorithm~\ref{alg:overall} can be performed in time complexity 
  \[
  O\left(\{z^\omega(\dim M)^\omega+mnz^\omega\}\card{\II{m,n}}\right)
  \]
  where $\dim M$ is the total dimension of $M$.
\end{theorem}

\subsection{Interval selection heuristic}

An important complexity drawback of the above is the number of intervals we need to check.
  Given a module $M$, we need, in the worst case, to compute the multiplicities  with respect to all intervals, which is $\card{\II{m,n}}$ in number.
  The heuristics explained below do not change this worst case analysis, but
we can hope that not all intervals appear in the decomposition for particular cases.
  Using adapted heuristics we can reduce the number of intervals to be checked.

\paragraph{\textbf{Contained-support heuristic}}
We note that if an interval is a summand of a module $M$ then its support is included in the support of $M$. Thus, the number of intervals to check can be reduced by only considering intervals included in the support of $M$.
For example, the algorithm \textproc{getCandidates}$(x,y)$ in Algorithm~\ref{alg:candidates_pointed} can be improved by including this heuristic, checking inclusion in the support of $\mathrm{dimVecRemaining}$ at each step, as $\mathrm{dimVecRemaining}$ represents the part of $M$ still unprocessed.

\paragraph{\textbf{Line-restriction heuristic}}
We can further reduce the number of intervals to be tested by the following heuristic, which builds up candidate intervals by stacking $1$D intervals to form $2$D intervals, with the $1$D intervals obtained by decomposing the restriction of $M$ to horizontal lines.

  Suppose that $M = \bigoplus_{i=1}^\ell T_i \oplus X$, where $T_i$ are all interval representations and $X$ has no interval summands. That is, $\bigoplus_{i=1}^\ell T_i$ is the interval-decomposable part of $M$. We wish to create a set of candidate intervals containing $\{T_i \mid i=1,\hdots,\ell\}$, without knowing the decomposition given above.

  The restriction of the above to a horizontal line $L$ on the commutative grid gives
  $M|_L = \bigoplus_{i=1}^\ell T_i|_L \oplus X|_L$, a decomposition of the $1$D persistence module $M|_L$.
  Note that since $T_i$ are $2$D interval representations, they restrict to $1$D interval representations $T_i|_L$. The indecomposable decomposition of $M|_L$ necessarily contains all of these $1$D intervals (together with intervals coming from $X|_L$). Note that since $M|_L$ is a $1$D persistence module, it decomposes into a set of $1$D intervals which we denote as
  $C(L)$.

  We stack valid combinations of intervals from $C(L)$ over all horizontal lines $L$, to produce a set of $2$D intervals that necessarily contain $\{T_i \mid i=1,\hdots,\ell\}$. By a valid stacking, we mean that the resulting $2$D object should be a valid interval, i.e. one with staircase shape. 
  This procedure is described in Algorithm~\ref{alg:combi}.

\begin{algorithm}[h]
  \caption{Combinatorial combination of line restrictions}
  \label{alg:combi}
  \begin{algorithmic}[1]
    \Require {$2$D persistence module $M$ over $Q=\Gf{m,n}$}
    \Function{StackingPotentialIntervals}{$M$}
    \State \textbf{initialize} $\mathcal{P}=\emptyset$ and $\mathcal{L}=\{0\}$, where $0$ is the zero module on the grid up to the $0^{th}$ line.
    \For{$i = 1, \dots , n$}
      \State $L_i \gets $ line at height $i$.
      \State $C(L_i) \gets $ intervals in decomposition of $M|_{L_i}$.
        \State $\mathcal{L}' \gets \{0\}$, where $0$ is the zero module on the grid up to the $i^{th}$ line.\label{algline:combizero}
        \For{$T\in\mathcal{L}$}
            \For{$S\in C(L_i)$}
                \If{$S$ on $T$ is a valid stacking}
                    \State Add the stacking of $S$ on $T$ to $\mathcal{L}'$.\label{algline:combistack}
                    \State Add the stacking of $S$ on $T$ extended by $0$ on the whole grid to $\mathcal{P}$.\label{algline:combifinish}
                \EndIf
            \EndFor
        \EndFor
        \State $\mathcal{L}\leftarrow\mathcal{L}'$
    \EndFor
    \State \Return $\mathcal{P}$
    \EndFunction
  \end{algorithmic}
\end{algorithm}

\begin{lemma}
    Algorithm~\ref{alg:combi} returns a superset of the intervals appearing in the decomposition of $M$.
\end{lemma}

\begin{proof}
    We successively build interval modules defined up to the $i$th line of the grid and extend upwards by stacking. 
    At each stage, one copy of the result is stored in $\mathcal{L}'$ and later sent to $\mathcal{L}$ for further stacking (line~\ref{algline:combistack}), while another copy is finished (sent to $\mathcal{P}$) by adding zeros to the rest of the grid (line~\ref{algline:combifinish}).
    
    The stacking operation consists of taking such a module $T \in \mathcal{L}$ defined up to the $(i-1)$th line and adding the interval $S$ (at the $i$th line) on top of it.
    If $T$ is not the $0$ module, then by construction the top line of $T$ is an interval $\intv[a,b]$. The considered $S \in C(L_i)$ is always some $1$D interval $\intv[c,d]$.
    The stacking of $S$ on $T$ forms a valid $2$D interval only when $c\leq a\leq d\leq b$ in order to have a staircase support. Then we form the module $T\rightarrow S$, where the arrow supports the canonical map between the two intervals. 
    By construction, this is a $2$D interval on the grid defined up to the $i$th line.
    On the other hand, if $T$ is $0$, then the module $T\rightarrow S$ is also well defined and is simply the $2$D interval that is nonzero only on the $i$th line.
    
    
    We need to check that all $2$D interval modules that are part of the decomposition of $M$ are in $\mathcal{P}$.
    This is a direct consequence of the following observation. 
    
    If
    $M = \bigoplus_{j=1}^\ell T_j \oplus X$, where $\bigoplus_{j=1}^\ell T_j$ is the interval-decomposable part of $M$,
    then for each $j$ the restriction $T_j$ to the $i$th line $T_j|_{L_i}$
    is an interval in $C(L_i)$.
    Moreover, $T_j$ can be rewritten as the stacking 
    $T_j|_{L_1} \rightarrow \dots \rightarrow T_j|_{L_n}$.
    Since $T_j$ is connected, the $1$D intervals $T_j|_{L_i}$ that are nonzero correspond to a contiguous sequence of indices $i$ (line heights), say $s \leq i \leq t$ for some $s$ and $t$.
    Thus, $T_j$ is formed in Algorithm~\ref{alg:combi}
    using the zero module up to the $(s-1)$th line (Line~\ref{algline:combizero} of Algorithm~\ref{alg:combi} with $i=s-1$), stacking $T_j|_{L_i} \in C(L_i)$ for $s \leq i \leq t$ (always valid stackings), and finishing up by zeros to the rest of the grid. 
    That is $T_j \in \mathcal{P}$ at the end of the algorithm.

\end{proof}

\paragraph{\textbf{Image-based heuristic}}

Another approach, given in Algorithm~\ref{alg:image}, is to use the ranks of maps to choose which intervals to test.
We start from $L$ the zero module and iteratively add interval summands of $M$ with their multiplicity to $L$.
At the end, if $M$ is interval-decomposable, then $M\cong L$.
We will greedily work towards equalizing the dimension vectors of $M$ and $L$.
If we reach a point where the greedy procedure fails, this means that the module $M$ is not interval-decomposable.

In lines~\ref{algline:bigS}~and~\ref{algline:smalls}, the algorithm selects the leftmost vertex $s$ on the lowest possible line where the dimensions of the vector spaces of $L$ and $M$ disagrees.
We then look for a rectangle $B$ as large as possible that must be contained in the support of at least one indecomposable summand of $M$ that does not yet appear in $L$.
In lines~\ref{algline:bigT}~and~\ref{algline:smallt}, we achieve this by selecting a maximal element $t$ such that the rank of the map $M(s\rightarrow t)$ is greater than the rank of $L(s\rightarrow t)$.
Then $B$ is a subset of each support of the intervals we want to find.

To reduce the number of candidates, we remark that those intervals must interact with the map $M(s\rightarrow t)$ in the following way.
We first initialize $F = M_s$, and progressively 
consider smaller and smaller subspaces of $F$ as we process intervals.
At each iteration of the inner while loop, we consider the subspace not accounted for previously, via a complementary basis in $F$ of the kernel of $M(s\rightarrow t)$ in line~\ref{algline:compbasis}. We note that we exclude the kernel because we want the intervals that contain the rectangle $B$ from $s$ to $t$ in its support.
Then the supports of the intervals of interest must be contained in the set of vertices reachable by images and pre-images along walks starting at those basis elements. This is encoded in the sets $C_i$ as defined in line~\ref{algline:ci}.
We can compute each set $C_i$ independently, and the intervals must be contained in the union $\bigcup_{i=k+1}^f C_i$.

Having obtained candidate intervals, in line~\ref{algline:mults} we then compute their true multiplicities $d_M(I)$ in $M$, for example via Algorithm~\ref{alg:multiplicity} as discussed above.
If $M$ is interval-decomposable, then 
\[
\sum\limits_{B \subset I \subset \bigcup C_i} d_M(I)
=
\rk M(s\rightarrow t)-\rk L(s\rightarrow t)
\]
since the intervals $I$  being considered (which contain the rectangle from $s$ to $t$), together with the already-processed $L(s\rightarrow t)$ should account for all the rank of $M(s\rightarrow t)$.
Thus, in line~\ref{algline:earlystop} we use this condition to determine whether or not to stop early.
If we are not sure that $M$ is not interval-decomposable, in line~\ref{algline:update} we update $L$ and $F$ and continue with the iteration.

\begin{algorithm}[h]
  \caption{Image based decomposition}
  \label{alg:image}
  \begin{algorithmic}[1]
    \Require {A module $M$}
    \Function{imageBasedDecomposition}{$M$}
    \State $L  \gets 0$.
    \While{$\dimv M \neq\dimv L$}
        \State $S \gets \{u\in\mathrm{Z}^2\mid (\dimv M-\dimv L)_u\neq 0\}$. \label{algline:bigS}
        \State\label{st:spick} $s$ $\gets$ the minimal element of $S$ on lowest row. \label{algline:smalls}
        \State $F \gets M_s$.
        \While{$\dim M_s\neq \dim L_s$}
            \State $T \gets \{u\in \mathrm{Z}^2\mid \rk M(s\rightarrow u)\neq \rk L(s\rightarrow u)\}$. \label{algline:bigT}
            \State $t$ $\gets$ a maximal element of $T$. \label{algline:smallt}
            \State Compute $\{x_1,\hdots, x_k, x_{k+1},\hdots, x_f\}$ a basis of $F$ so that 
            $\{x_i\}_{i=1}^k$ is a basis of $F \cap \Ker M(s \rightarrow t)$. \label{algline:compbasis}
            \State $B$ $\gets$ the rectangle with lower left corner $s$ and upper right corner $t$.
            \ForAll{$i = k+1, \hdots, f$}
            \State
            $C_i \gets \left\{
              r \;\middle|\;
              \begin{array}{l}
                \exists (r_l)_{l\in[0,j]} \text{ such that } \forall 0<l<j-3, r_{l} \text{ and } r_{l+2}
                \text{ are incomparable},\\
                r_{0}=s,\ r_1=t,\ r_j=r, \text{ and } \forall l,\ r_l 
                \text{ does not lie below the $s$th line},\\
                M(r_{j}\rightarrow r_{j-1})^{-1}M(r_{j-1}\rightarrow r_{j-2})\cdots M(r_2\rightarrow r_1)^{-1}M(r_0\rightarrow r_1)x_i\neq \{0\}
              \end{array}
              \right\}$.\label{algline:ci}
            \EndFor
            \State Compute $d_M(I)$ for every interval module $I$ such that $B\subset\Supp I\subset \bigcup_{i=k+1}^f C_i$. \label{algline:mults}
            \State\label{st:testdeco} If 
            $\sum\limits_{B \subset I \subset \bigcup C_i} d_M(I)\neq \rk M(s\rightarrow t)-\rk L(s\rightarrow t)$, return that $M$ is not interval-decomposable. \label{algline:earlystop}
            \State $L \gets L \oplus \bigoplus\limits_{B \subset I \subset \bigcup C_i} I^{(d_M(I))}$ and $F \gets F\cap \Ker M(s\rightarrow t)$. \label{algline:update}     
        \EndWhile
    \EndWhile
    \State \Return $L$, the interval decomposition of $M$.
    \EndFunction
  \end{algorithmic}
\end{algorithm}

\begin{lemma}
If $M$ is interval-decomposable then Algorithm \ref{alg:image} returns the interval decomposition of $M$.
Otherwise Algorithm \ref{alg:image} stops and indicates that $M$ is not interval-decomposable.
\end{lemma}

\begin{proof}
    First, if the algorithm returns a module $L$ then it is necessarily the interval decomposition of $M$.
    Indeed, by construction, $L$ is a direct sum of intervals, and every such interval appears exactly the same number of times in  $L$ as it does in $M$. So $M$ is isomorphic to the direct sum of $L$ with another module $L'$.
    Moreover, $\dimv M=\dimv L$ if $L$ is returned, and so $L'=0$ and $M\cong L$.
    
    Second, we need to show that the algorithm always terminates.
    If $\dimv M\neq\dimv L$ then $S\neq 0$ and $s$ is well defined.
    For the second while loop, if 
    $\dim M_s\neq \dim L_s$ 
    then $T$ is not empty because $M(s\rightarrow s)$ is the identity with rank equal to $\dim M_s\neq \dim L_s$, and
    so $s\in T$.
    Thus, a maximal element $t$ of $T$ exists.
    Therefore, for every round of this loop, either the algorithm returns that $M$ is not interval-decomposable, or the dimension $F$ decreases and $T$ is reduced.
    
    Finally, we must show that if $M$ is interval-decomposable, Algorithm~\ref{alg:image} will always return a module $L$.
    Assume that $M$ is interval-decomposable.
    Arriving at line~\ref{st:testdeco}, we have the following properties.
    We have picked two elements $s$ and $t$, and we already know an interval decomposable module $L$ that appears in the decomposition of $M$.
    As $M$ is interval-decomposable, we have at most 
    $\rk M(s\rightarrow t)$ distinct interval modules $I_1, \hdots, I_r$
     in the decomposition of $M$ such that $\rk I_i(s\rightarrow t)=1$. Note that $\sum_{i=1}^r d_M(I_i)=\rk M(s\rightarrow t)$.
    
    We  separate the set $\{I_1, \hdots, I_r\}$ in two sets depending on their lowest left corner, i.e. the leftmost vertex on the lowest line of its support.
    \begin{enumerate}
    
    \item If $I\in \{I_1, \hdots, I_r\}$ has a lowest left corner $s'\neq s$ then either $s'$ is on the left of $s$ or on a line below $s$.
    In both cases, due to the choice of $s$ at line~\ref{st:spick}, we have $\dimv M_{s'} = \dimv L_{s'}$. 
    By construction,
    $L$ is a summand of $M$, hence for every $J$ such that $J_{s'}\neq 0$, $d_L(J)=d_M(J)$. 
    In particular, $d_L(I)=d_M(I)$.
    
    \item If $I\in \{I_1, \hdots, I_r\}$ has lowest left corner $s$, we consider $t_1,\hdots,t_j$ the set of all $t$ that have been considered since $s$ has been picked, excluding the current $t$.
    An interval $I\in \{I_1, \hdots, I_r\}$ can fall into two categories.
    If there exists $1\leq k\leq j$ such that $\rk I(s\rightarrow t_k)\neq 0$ then $I$ has been considered when considering the first such $k$ and added to $L$ with the correct multiplicity, i.e. $d_L(I)=d_M(I)$.
    
    Otherwise, $I_s\subset \bigcap_{i=1}^j \Ker M(s\rightarrow t_j)= F$. 
    Let us show that the support of $I$ is included in $\bigcup C_i$.
    As $M$ is interval decomposable and $I$ is a summand of $M$, say $M = I \oplus N$, there exists a choice of basis for $M_u$ such that $M_u=I_u\oplus N_u$ for every vertex $u \in (\Gf{m,n})_0$.
    In particular $M_s=I_s \oplus N_s$.
    Let $0 \neq x\in I_s$.
    The support of $I$ is connected, so for every $v$ in the support of $I$, there exists a walk $(v_1,\dots, v_p)$ of elements of the support of $I$ such that $v_q$ and $v_{q+1}$ are adjacent, $v_1=t$ and $v_p=v$.
    Furthermore, the map $I(v_q\rightarrow v_{q'})$ is an isomorphism if $v_q\leq v_{q'}$.
    
    As the support of $I$ is convex and contains no element below the line of $s$, we can build an alternating walk of paths by extracting a subsequence $(r_l)_{l=1}^j$ of $(v_1,\dots, v_p)$ such that 
    no $r_l$ lies on a line lower than $s$, 
    $r_{l}$ and $r_{l+2}$ are incomparable for $l<j-3$, 
    $r_1=t$ and $r_j=v$.
    Then fix $r_0=s$.
    For ease of notation, we require that $j$ is even.
    If the length is odd, we simply put $r_{j+1}=r_j$ and use the subsequence until $r_{j+1}$.
    
    The map $I(r_{j}\rightarrow r_{j-1})^{-1}I(r_{j-1}\rightarrow r_{j-2})\cdots I(r_2\rightarrow r_1)^{-1}I(r_0\rightarrow r_1)$ is an isomorphism.
    Translated into $M$, this implies that 
    \[
    M(r_{j}\rightarrow r_{j-1})^{-1}M(r_{j-1}\rightarrow r_{j-2})\cdots M(r_2\rightarrow r_1)^{-1}M(r_0\rightarrow r_1)(x)\neq \{0\}.
    \]
    As $x$ can be expressed as a linear combination of $\{x_{1},\dots, x_f\}$, there exists at least one $i\in \{1, \dots, f\}$ such that 
    \[
    M(r_{j}\rightarrow r_{j-1})^{-1}M(r_{j-1}\rightarrow r_{j-2})\cdots M(r_2\rightarrow r_1)^{-1}M(r_0\rightarrow r_1)(x_i)\neq \{0\}.
    \]
    Since all elements $x_j$ for $j\leq k$ are part of $\Ker M(s\rightarrow t)$, we have $i\geq k+1$ and $v\in C_i$.
    
    Moreover, $\rk I(s\rightarrow t)\neq 0$ and thus the support of $I$ also contains the rectangle $B$.
    Note also that there is no double counting.
    Therefore $\rk M(s\rightarrow t)=\rk L(s\rightarrow t)+\sum d_M(I)$, and the algorithm does not incorrectly stop early.
    \end{enumerate}

\end{proof}

\subsection{Interval decomposability with~\cite{dey2019generalized} decomposition algorithm}
\label{subsec:decompo}

Dey and Xin proposed in~\cite{dey2019generalized} a generalization of the persistence algorithm to decompose multidimensional persistence modules.
Once an indecomposable decomposition is computed, testing for interval-decomposability is a simple matter as one only needs to check that all elements of the decomposition are interval modules.
The generalized persistence algorithm however is limited to a specific case:
In the matrix encoding the minimal presentation of the module, no pair of columns nor pair of rows can have the same grade.
Translated into the language of this paper, this means that the generators of the projective modules appearing in the minimal projective presentation of the module are all distinct.

It was suggested in~\cite{dey2019generalized} that for modules not satisfying this property, an easy workaround can be implemented.
In this case, the generalized persistence algorithm does not always provide a full decomposition.
It nonetheless returns a direct sum decomposition of the module,
the only limitation being that some of the summands might not be indecomposable.
The suggested workaround is to arbitrarily fix an order on the rows and columns which have same grades, in essence artificially breaking ties in the grades.
By exhaustively checking all such tie-breaking orders on the rows and columns, it was claimed that the algorithm will provide a full decomposition at some point.

This workaround is valid if there exists an order that allows for a full decomposition. Unfortunately, this is not always true as we show with the following example. We show that for whatever order chosen for elements with the same grade, no full decomposition can be obtained through the algorithm from~\cite{dey2019generalized}.
In fact, we show a stronger statement, that for whatever order chosen for elements with the same grade, no algorithm using only ``matrix operations in one direction'' can obtain a full decomposition.

\begin{example}
  We consider the field $K = \mathbb{F}_{2^2} = \mathbb{F}_2(\lambda) = \{0,1,\lambda,\lambda^2\}$, where $\lambda$ satisfies $\lambda^2 + \lambda + 1 = 0$.
  Let
  \[
    \begin{tikzcd}[ampersand replacement=\&,graphstyle]
      \& c \ar[rr] \& \& \omega \\
      z \ar[ur] \ar[rr] \& \& b \ar[ur] \& \\
      \& y \ar[uu] \ar[rr] \& \& a \ar[uu] \\
      \alpha \ar[uu] \ar[ur] \ar[rr] \& \& x \ar[uu] \ar[ur] \&
    \end{tikzcd}
    \;\;\text{ and }\;\;
    M:
    \begin{tikzcd}[ampersand replacement=\&,graphstyle]
      \& K^2 \ar[rr] \& \& 0 \\
      K^2 \ar[ur] \ar[rr] \& \& K^2 \ar[ur] \& \\
      \& K^2 \ar[uu] \ar[rr] \& \& K^2 \ar[uu] \\
      0 \ar[uu] \ar[ur] \ar[rr] \& \& K^2 \ar[uu] \ar[ur,swap]{}
      {
        \left[\begin{matrix}
          0 & 1 \\ 1 & 1
        \end{matrix}\right]
      } \&
    \end{tikzcd}
  \]
 be the $2\times 2\times 2$ commutative cube $\Gf{2,2,2}$ and a $K$-representation of $\Gf{2,2,2}$, respectively.

  Then one can calculate the following minimal projective presentation for $M$:
  \[
    \begin{tikzcd}
      P(a)^2 \oplus P(b)^2 \oplus P(c)^2
      \rar{p_1}
      &
      P(x)^2 \oplus P(y)^2 \oplus P(z)^2
      \rar{p_0}
      &
      M
      \rar
      &
      0
    \end{tikzcd}
  \]
  where $P(v)$ is the indecomposable projective representation with source $v$, and where the morphism $p_1$ can be given in matrix form as
  \begin{equation}
    \label{eq:p1_cube}
    p_1 : \;\;
    \begin{blockarray}{ccc|cc|cc}
      &
      \BAmulticolumn{2}{c|}{{\scriptstyle P(a)\mathrlap{{}^2}}}
      &
      \BAmulticolumn{2}{c|}{{\scriptstyle P(b)\mathrlap{{}^2}}}
      &
      \BAmulticolumn{2}{c}{{\scriptstyle P(c)\mathrlap{{}^2}}}
    \\
    \begin{block}{c[cc|cc|cc]}
      \multirow{2}{*}{${\scriptstyle P(x)^2}$}
      & 1 & 1 & 1 & 0 & 0 & 0 \topstrut
      \\
      & 1 & 0 & 0 & 1 & 0 & 0
      \\
      \BAhline
      \multirow{2}{*}{${\scriptstyle P(y)^2}$}
      & 1 & 0 & 0 & 0 & 1 & 0
      \\
      & 0 & 1 & 0 & 0 & 0 & 1
      \\
      \BAhline
      \multirow{2}{*}{${\scriptstyle P(z)^2}$}
      & 0 & 0 & 1 & 0 & 1 & 0
      \\
      & 0 & 0 & 0 & 1 & 0 & 1 \botstrut
      \\
    \end{block}
    \end{blockarray}.
  \end{equation}

  First, let us show that $M$ is decomposable.
  We note that vertices $x$, $y$, and $z$ do not have any arrows between them, and similarly for $a$, $b$, $c$. Thus, allowable matrix operations are restricted to within each block row or column
  in Equation~\eqref{eq:p1_cube}. For ease of notation, let $X$ be the matrix
  $
    X =
    \left[
      \begin{array}{cc}
        1 & 1 \\
        1 & 0
      \end{array}
    \right].
  $
  Then $p_1$ transforms as
  \[
    \begin{array}{rcl}
      p_1 :
      \left[
      \begin{array}{c|c|c}
        X & I & 0 \\ \hline
        I & 0 & I \\ \hline
        0 & I & I
      \end{array}
    \right]
               &
                \cong
                &
    \left[
      \begin{array}{c|c|c}
        PX & P & 0 \\ \hline
        I & 0 & I \\ \hline
        0 & I & I
      \end{array}
    \right]
    \cong
    \left[
      \begin{array}{c|c|c}
        PX & PP^{-1} & 0 \\ \hline
        I & 0 & I \\ \hline
        0 & P^{-1} & I
      \end{array}
    \right]
    \cong
    \left[
      \begin{array}{c|c|c}
        PX & PP^{-1} & 0 \\ \hline
        I & 0 & I \\ \hline
        0 & PP^{-1} & P
      \end{array}
    \right]
     \\ \\
      &
      \cong
      &
      \left[
      \begin{array}{c|c|c}
        PX & PP^{-1} & 0 \\ \hline
        I & 0 & P^{-1} \\ \hline
        0 & PP^{-1} & PP^{-1}
      \end{array}
    \right]
    \cong
    \left[
      \begin{array}{c|c|c}
        PX & PP^{-1} & 0 \\ \hline
        P & 0 & PP^{-1} \\ \hline
        0 & PP^{-1} & PP^{-1}
      \end{array}
    \right]
\cong
\left[
      \begin{array}{c|c|c}
        PXP^{-1} & PP^{-1} & 0 \\ \hline
        PP^{-1} & 0 & PP^{-1} \\ \hline
        0 & PP^{-1} & PP^{-1}
      \end{array}
    \right]
\\ \\ &=&
\left[
      \begin{array}{c|c|c}
        PXP^{-1} & I & 0 \\ \hline
        I & 0 & I \\ \hline
        0 & I & I
      \end{array}
    \right]
    \end{array}
  \]
  for $P$ any invertible $2\times 2$ matrix, by alternating row and column operations with respect to $P$.
  That is, the matrix form of $p_1$ can be transformed by conjugation of $X$, without affecting the other block entries.

  By letting
  \[
    P =
    \left[
      \begin{array}{cc}
        1 & \lambda \\
        \lambda & 1
      \end{array}
    \right]
    \text{, we have }
    P^{-1} =
    \left[
      \begin{array}{cc}
        \lambda^2 & 1 \\
        1 & \lambda^2
      \end{array}
    \right]
  \]
  since $\lambda^2 + \lambda = -1 = 1$ and $\lambda^3+1 = 0$ in the base field $K= \mathbb{F}_2(\lambda)$. Thus
  \[
    PXP^{-1} =
    \left[
      \begin{array}{cc}
        1 & \lambda \\
        \lambda & 1
      \end{array}
    \right]
    \left[
      \begin{array}{cc}
        1 & 1 \\
        1 & 0
      \end{array}
    \right]
    \left[
      \begin{array}{cc}
        \lambda^2 & 1 \\
        1 & \lambda^2
      \end{array}
    \right]
    =
    \left[
      \begin{array}{cc}
        1 + \lambda & 1 \\
        1 + \lambda & \lambda
      \end{array}
    \right]
    \left[
      \begin{array}{cc}
        \lambda^2 & 1 \\
        1 & \lambda^2
      \end{array}
    \right]
    =
    \left[
      \begin{array}{cc}
        \lambda^2 & 0 \\
        0 & \lambda
      \end{array}
    \right].
  \]
  The above computations show that
  \begin{equation}
    \label{eq:p1answer}
    p_1 \cong \;\;
    \begin{blockarray}{ccc|cc|cc}
      &
      \BAmulticolumn{2}{c|}{{\scriptstyle P(a)\mathrlap{{}^2}}}
      &
      \BAmulticolumn{2}{c|}{{\scriptstyle P(b)\mathrlap{{}^2}}}
      &
      \BAmulticolumn{2}{c}{{\scriptstyle P(c)\mathrlap{{}^2}}}
    \\
    \begin{block}{c[cc|cc|cc]}
      \multirow{2}{*}{${\scriptstyle P(x)^2}$}
      & \lambda^2 & 0 & 1 & 0 & 0 & 0 \topstrut
      \\
      & 0 & \lambda & 0 & 1 & 0 & 0
      \\
      \BAhline
      \multirow{2}{*}{${\scriptstyle P(y)^2}$}
      & 1 & 0 & 0 & 0 & 1 & 0
      \\
      & 0 & 1 & 0 & 0 & 0 & 1
      \\
      \BAhline
      \multirow{2}{*}{${\scriptstyle P(z)^2}$}
      & 0 & 0 & 1 & 0 & 1 & 0
      \\
      & 0 & 0 & 0 & 1 & 0 & 1 \botstrut
      \\
    \end{block}
  \end{blockarray}
  =
  \left[
    \begin{array}{c|c|c}
      \lambda^2 & 1 & 0 \\ \hline
      1 & 0 & 1 \\ \hline
      0 & 1 & 1
    \end{array}
  \right]
  \oplus
  \left[
    \begin{array}{c|c|c}
      \lambda & 1 & 0 \\ \hline
      1 & 0 & 1 \\ \hline
      0 & 1 & 1
    \end{array}
  \right],
\end{equation}
with the two summands each inducing a nontrivial summand of $M \cong \Coker p_1 \cong \Coker p_1' \oplus \Coker p_1''$, where $p_1'$ ($p_1''$, respectively) is the first (second, respectively) direct summand of $p_1$ in the decomposition above.

Next, let us consider the workaround proposed by Dey and Xin, for when generators of the projectives appearing in the projective presentation of $M$ have equal grades. This is exactly the case we have here, as we have each two copies of $P(v)$ for $v=a,b,c$ and $v=x,y,z$.

In general, without any restrictions, $p_1$ given in Equation~\eqref{eq:p1_cube}
can be transformed into
\begin{equation}
  \label{eq:p1trans}
  \begin{array}{l}
    \left[\begin{array}{c|c|c}
            A_1 1_{P(x)}& 0 & 0 \\ \hline
            0 & A_2 1_{P(y)} & 0 \\ \hline
            0 & 0 & A_3 1_{P(z)}
    \end{array}\right]
    \left[\begin{array}{c|c|c}
            X & I & 0 \\ \hline
            I & 0 & I \\ \hline
            0 & I & I
    \end{array}\right]
    \left[\begin{array}{c|c|c}
            B_1 1_{P(a)} & 0 & 0 \\ \hline
            0 & B_2 1_{P(b)} & 0 \\ \hline
            0 & 0 & B_3 1_{P(c)}
    \end{array} \right]
    \\
    =
    \left[\begin{array}{c|c|c}
      A_1 X B_1 & A_1 B_2 & 0 \\ \hline
      A_2 B_1  & 0 & A_2 B_3 \\ \hline
      0 & A_3 B_2 & A_3 B_3
    \end{array} \right]
  \end{array}
\end{equation}
where $X$ is as defined above, $A_1$, $A_2$, $A_3$, $B_1$, $B_2$, $B_3$ are invertible
$2 \times 2$ $K$-matrices, and $1_{P(v)}$ is the identity morphism of $P(v)$. Note that since there are no nonzero morphisms among $P(x)$, $P(y)$,  $P(z)$, and among $P(a)$, $P(b)$, $P(c)$, the off-diagonal blocks are always zero.

The workaround involves arbitrarily fixing an order for the rows and columns which have the same grades, and running their algorithm. Their algorithm only performs row/column operations in ``one direction'' with respect to the fixed order.
This involves restricting $A_1$, $A_2$, $A_3$, $B_1$, $B_2$, $B_3$ in the transformation matrices to either being upper or lower triangular. We note that we do not a priori require the matrices to be all upper or all lower.
We show that it is impossible to compute a decomposition for $p_1$ with this restriction.

First, let us study the product $AB$ of two upper or lower invertible $2 \times 2$ matrices $A$ and $B$,
since $p_1$ after transformation (Equation~\eqref{eq:p1trans})
contains blocks of that form.
Let
$
A =
\left[
  \begin{array}{cc}
    a_1 & a_u \\
    a_l & a_2
  \end{array}
\right]
$
and
$
B =
\left[
  \begin{array}{cc}
    b_1 & b_u \\
    b_l & b_2
  \end{array}
\right]
$.
Since we impose that $A$ be upper or lower triangular and invertible,
we have $a_u = 0$ or $a_l = 0$, and $a_1 \neq 0, a_2 \neq 0$, with similar
conditions on $B$.

Furthermore, we know \emph{one particular} decomposition of $p_1$ as given in Equation~\eqref{eq:p1answer}, with each block row and block column decomposing into two. Thus, any other nontrivial decomposition of $p_1$ must have its blocks of the form $AB$ be diagonal matrices. Note that given the restrictions on $A$ and $B$, $AB$ cannot be anti-diagonal.

Requiring the diagonality of 
\[
  AB =
  \left[
    \begin{array}{cc}
      a_1 b_1 + a_u b_l & a_1 b_u + b_2 a_u \\
      b_1 a_l + a_2 b_l & a_2 b_2 + a_l b_l
    \end{array}
  \right]
\] 
is equivalent to requiring that $a_1 b_u + b_2 a_u = 0$ and $b_1 a_l + a_2 b_l = 0$. Since $a_1,a_2,b_1,b_2$ are all nonzero, we conclude that $a_u = 0$ if and only if $b_u = 0$, and $a_l = 0$ if and only if $b_l = 0$. That is, the ``shape''  (upper or lower) of $A$ is the same as the ``shape'' of $B$.

The transformed $p_1$ in Equation~\eqref{eq:p1trans} has blocks $A_1B_2$, $A_2B_1$, $A_2B_3$, $A_3B_2$, and $A_3B_3$. Requiring that they all be diagonal implies that the shapes of $A_1$, $A_2$, $A_3$, $B_1$, $B_2$, $B_3$ are all the same. That is, the transformation blocks need to all be upper triangular, or all lower triangular.

Finally, we consider the final block $A_1 X B_1$ in Equation~\eqref{eq:p1trans},
where $X =
\left[
  \begin{array}{cc}
    1 & 1 \\
    1 & 0
  \end{array}
\right]
$ as before.
Suppose that all the transformation blocks are upper triangular.
In particular, $A_1$ and $B_1$ are upper triangular ($a_l = 0$, $b_l = 0$).
Then
\[
  A_1 X B_1 =
  \left[
    \begin{array}{cc}
      a_1 b_1 + b_1 a_u & a_1 b_2 + a_1 b_u + b_2 a_u \\
      a_2 b_1 & a_2 b_u
    \end{array}
  \right].
\]
Since $a_2 b_1$ cannot be zero, this cannot be diagonal.
Similarly, in the case that all the transformation blocks are lower triangular,
$A_1 X B_1$ cannot be diagonal. By the arguments above, there are no other possibilities for obtaining a nontrivial decomposition of $p_1$. Thus, given the restrictions on the transformations on $p_1$, one cannot obtain a full decomposition of $p_1$.
\end{example}

\section{Conclusion}
In this paper we presented an algorithm for testing $\mathcal{S}$-decomposability for any finite set $\mathcal{S}$ of indecomposables, based on the procedure \cite{Asashiba2017} for computing the multiplicity of a given indecomposable in the decomposition of a module.
We specifically studied the case of interval-decomposability by first providing a characterization and an enumeration method for interval modules in the $2$D equioriented commutative grid case.
To the extent of our knowledge, this is the first algorithm to test interval-decomposability of a module, without the need for computing its full decomposition if the answer is negative.

Interval modules have a very specific structure that made computation easier, especially the fact that their endomorphism rings are isomorphic to the underlying field $K$. This slightly simplified Algorithm~\ref{alg:ass-p}, an essential component of the algorithm, compared to the general procedure (see Section 3.2 of \cite{gabriel1980auslander}).
When considering a different class $\mathcal{S}$ of indecomposables, the aforementioned simplification in Algorithm~\ref{alg:ass-p} may no longer hold, but the general procedure is still valid.

Another generalization is to consider interval modules of $n$D commutative grids with $n > 2$. More generally, in any finite bound quiver, we can still define and enumerate interval modules by using a brute-force approach.  Then we can apply our interval-decomposability algorithm. However, the brute-force enumeration comes with an additional cost as we do not have an easy characterization of intervals in the general case, in contrast to the staircase shape of $2$D intervals.

For the case of $2$D equioriented  commutative grids considered in this paper, we also provided several heuristics to try to speed up the enumeration of interval modules and the testing of interval-decomposability.
It would be interesting to implement these various heuristics and conduct an in-depth comparison on practical instances to evaluate their performances.



\section*{acknowledgements}
  We would like to thank Tamal Dey and Cheng Xin for bringing to our attention the question of determining interval-decomposability.
  \\~\\
  On behalf of all authors, the corresponding author states that there is no conflict of interest.

\bibliographystyle{plain}
\bibliography{refs}

\end{document}
